\documentclass{amsart}

\usepackage{amsfonts,amsmath,amssymb,amstext,amsthm}
\usepackage{epic}
\usepackage{graphicx,color}
\usepackage{mathrsfs}

\newcommand{\nwc}{\newcommand}

\nwc{\aaa}{\mathcal{F}}
\nwc{\aab}{\bar{\mathfrak{a}}}
\nwc{\aal}{\mathcal{F}'}
\nwc{\aap}{\mathcal{F}_{P}}

\nwc{\bbb}{\mathfrak{b}}
\nwc{\bbp}{\mathfrak{b}_{P}}
\nwc{\be}{\begin{equation}}
\nwc{\bea}{\begin{eqnarray}}
\nwc{\beq}{$$}

\nwc{\C}{\;\mbox{{\sf I}}\!\!\!C}
\nwc{\cb}{\overline{C}}
\nwc{\ccc}{\mathcal{C}}
\nwc{\cin}{\textbf{(v)}}
\nwc{\cl}{C'}
\nwc{\cp}{\mathcal{C}_{P}}
\nwc{\cpll}{\mathcal{C}_{P'}}
\nwc{\ct}{\tilde{C}}

\nwc{\dd}{\mathcal{L}}
\nwc{\ddd}{\mathfrak{d}}
\nwc{\ddl}{\mathcal{L}'}
\nwc{\dlp}{\delta_{P}}
\nwc{\doi}{\textbf{(ii)}}

\nwc{\ee}{\end{equation}}
\nwc{\eea}{\end{eqnarray}}
\nwc{\enq}{$$}

\nwc{\gon}{{\rm gon}}
\nwc{\gtl}{\tilde{g}}
\nwc{\gud}{g^{1}_{2}}
\nwc{\gtu}{g^{1}_{3}}

\nwc{\hhza}{H^{0}(C,\mathfrak{a})}
\nwc{\hua}{h^{1}(C,\mathfrak{a})}
\nwc{\hza}{h^{0}(C,\mathfrak{a})}

\nwc{\kk}{{\rm K}}

\nwc{\lbd}{\lambda}
\nwc{\lif}{L_{\infty}}

\nwc{\mm}{\mathfrak{m}}
\nwc{\mmp}{\mathfrak{m}_{P}}
\nwc{\mpd}{{\mathfrak{m}_{P}}^{2}}

\nwc{\N}{I\!\!N}
\nwc{\nn}{\mathbb{N}}

\nwc{\obp}{\overline{\mathcal{O}}_P}
\nwc{\ocbux}{\oo _{\bar{C}}\langle 1,x\rangle}
\nwc{\oclux}{\oo _{C'}\langle 1,x\rangle}
\nwc{\ocux}{\oo _{C}\langle 1,x\rangle}
\nwc{\ol}{\mathcal{O}'}
\nwc{\oma}{\Omega (\mathfrak{a})}
\nwc{\omo}{\Omega (\mathcal{O})}
\nwc{\oo}{\mathcal{O}}
\nwc{\op}{\mathcal{O}_P}
\nwc{\opc}{\mathcal{O}_{P,C}}
\nwc{\oph}{\widehat{\mathcal{O}}_{P}}
\nwc{\opl}{\mathcal{O}_{P}'}
\nwc{\oplc}{\mathcal{O}_{P,C}'}
\nwc{\opll}{\mathcal{O}_{P'}}
\nwc{\opt}{\tilde{\mathcal{O}}_{P}}
\nwc{\optt}{{\mathcal{O}}_{\tilde{P}}}
\nwc{\oq}{\mathcal{O}_{Q}}
\nwc{\oqt}{\tilde{\mathcal{O}}_{Q}}
\nwc{\ot}{\tilde{\mathcal{O}}}
\nwc{\overop}{\bar{\oo}_{P}}

\nwc{\pb}{\overline{P}}
\nwc{\pgmd}{\mathbb{P}^{g+2}}
\nwc{\pgmu}{\mathbb{P}^{g+1}}
\nwc{\pp}{\mathbb{P}}
\nwc{\prv}{\noindent\textbf{Proof}:}
\nwc{\pt}{\tilde{P}}
\nwc{\ptl}{\tilde{P}}
\nwc{\pum}{\mathbb{P}^{1}}
\nwc{\carta}{\mathfrak{U}}

\nwc{\Q}{\;\mbox{{\sf I}}\!\!\!Q}
\nwc{\qtl}{\tilde{Q}}
\nwc{\qua}{\textbf{(iv)}}

\nwc{\R}{I\!\!R}

\nwc{\sep}{\beq\ast\ \ast\ \ast\enq}
\nwc{\spl}{{S_{P}}'}
\nwc{\spll}{S_{P'}}
\nwc{\ssp}{{\rm S}_{P}}
\nwc{\sss}{{\rm S}}
\nwc{\sys}{\mathcal{L}}

\nwc{\tre}{\textbf{(iii)}}

\nwc{\um}{\textbf{(i)}}

\nwc{\vlp}{\mathcal{V}_{\lambda,P}}
\nwc{\vpt}{v_{\ptl}}
\nwc{\vv}{\mathcal{W}}
\nwc{\vvp}{\mathcal{W}_{P}}

\nwc{\wol}{\ww\cdot\mathcal{O}'}
\nwc{\wpn}{{\omega _{P}}^{n}}
\nwc{\ww}{\omega}
\nwc{\wwp}{\omega _{P}}

\nwc{\Z}{{Z\!\!\!Z}}
\nwc{\zz}{\mathbb{Z}}
\newcommand{\Pic}{\operatorname{Pic}}

\newtheorem{coro}{Corollary}[section]
\newtheorem{defi}[coro]{Definition}

\newtheorem{lema}[coro]{Lemma}
\newtheorem{prop}[coro]{Proposition}
\newtheorem{rem}[coro]{Remark}
\newtheorem{teo}[coro]{Theorem}

\begin{document}

\title{Curves with Canonical Models on Scrolls}

\author{Danielle Lara}
\address{Departamento de Matem\'atica, UFV / CAF,
Rodovia LMG 818 km 06,
35690-000 Florestal MG, Brazil}
\email{danilara@ufv.br}

\author{Simone Marchesi}
\address{Instituto de Matem\'atica, Estat\'istica e Computa\c c\~ao Cient\'ifica, UNICAMP, 
Rua S\'ergio Buarque de Holanda 651,
Distr. Bar\~ao Geraldo,
13083-859 Campinas SP, Brazil}
\email{marchesi@ime.unicamp.br }

\author{Renato Vidal Martins}
\address{Departamento de Matem\'atica, Instituto de Ci\^encias Exatas, UFMG,
Av. Ant\^onio Carlos 6627,
30123-970 Belo Horizonte MG, Brazil}
\email{renato@mat.ufmg.br}



\begin{abstract}
Let  $C$ be an integral and projective curve whose canonical model $C'$ lies on a rational normal scroll $S$ of dimension $n$. We mainly study some properties on $C$, such as gonality and the kind of singularities, in the case where $n=2$ and $C$ is non-Gorenstein, and in the case where $n=3$, the scroll $S$ is smooth, and $C'$ is a local complete intersection inside $S$. We also prove that a rational monomial curve with just one singular point lies on a surface scroll  if and only if its gonality is at most $3$, and that it lies on a threefold scroll if and only if its gonality is at most $4$.
\end{abstract}

\maketitle


\section*{Introduction}

In the late 1910's, F. Enriques proved in \cite{En} that a nonhyperelliptic smooth canonical curve is the set theoretic intersection of hyperquadrics, unless it is trigonal or isomorphic to a plane quintic. Nearly two decades later, this result was also showed by D. W. Babbage in  \cite{B}. In the early 1920's, K. Petri obtained a thourough description of the canonical ideal in terms of equations. As presented in 
\cite[p. 131]{ACGH}, the so called ``Petri's analysys of the canonical ideal" follows Enriques' division into cases, and is based on M. Noether's dimension counts on the glogal sections of the canonical sheaf \cite{N}. So these statements on a canonical curve are sometimes known as Noether-Enriques-Babbage-Petri Theorem.

It can be read off from them a well known geometric characterization of trigonality: a nonhyperelliptic smooth integral and projective curve is trigonal if and only if it is isomorphic to a canonical curve which lies on a nonsingular two-dimensional rational normal scroll; moreover, the $g_3^1$ that computes the gonality of the curve is cut out by the ruling of the scroll.

Actually, the same statement above holds for (possibly singular) Gorenstein curves as well, just observing that the dualizing sheaf of any such a curve is a bundle. So, later on, in \cite{RS}, R. Rosa and K.-O. St\"ohr devoted their study to the case where the scroll is singular, i.e., a cone. It turned out that they got similar results which perfectly matched the known ones. The additional hypothesis they made was that the linear series could admit non-removable base points. This was because canonical (Gorenstein) curves lying on a cone necessarily pass through the vertex and the ruling of the cone cuts out a $g_3^1$. In other words, the vertex cannot be removed, otherwise the curve would be hyperelliptic.

A way of formalizing linear series with non-removable base points appears in M. Coppens \cite{Cp}. In \cite{S1}, St\" ohr established a manner of dealing with this condition, still keeping the language of divisors. Essentialy, one is allowing torsion free sheaves of rank $1$, rather than bundles, regarding linear systems; and replacing morphisms by pencils, concerning gonality. This discussion is more precise in Definition \ref{defsis}.

In \cite{M1} and \cite{M2}, the topic was led to the case where the curve needs not to be Gorenstein and, in particular, there is no canonical way of embedding it. So given any such a curve $C$, the study was made by means of its canonical model $C'$ defined in terms of the global sections of the dualizing sheaf of $C$ (Definition \ref{defcan}). It was noted that if $C$ is trigonal then $C'$ lies on a two-dimensional rational normal scroll (a \emph{surface scroll} for short), and some properties on $C$ were derived from this.

The point of departure of this article is precisely the converse of the above assertion, that is, assume $C$ is non-Gorenstein and $C'$ lies on a surface scroll $S$, then what can be said about $C$? In order to answer the question it is crucial to distinguish some important properties on non-Gorenstein curves to take into account. In \cite[Thms. 5.10, 6.5]{KM}, S. L. Kleiman along with the third named author characterized, by local invariants, curves for which the canonical model is projectively normal and arithmetically normal, called, respectively, \emph{nearly Gorenstein} and \emph{nearly normal}. So this turns into our roadmap, that is, to check when these conditions hold according to the intersection number $\ell$ of $C'$ and a fiber of $S$. The results we got are summarized in Theorem \ref{thmtwo}. Compared to \cite{M1} and \cite{M2}, we give refined statements with less hypothesis, i.e., just assuming that $C'\subset S$.

The second task was trying to find examples of non-Gorenstein curves with canonical models on surface scrolls, though not trigonal. As it is usual, we dealt with rational monomial curves and, instead, we got the following result: the canonical model of any such a curve lies on a surface scroll if and only if the curve is trigonal (Theorem \ref{scroll2}). The proof involves some combinatorics on semigroup of values. We took advantage from the arguments we used to do a little step forward into a potential moduli problem. The question is when nonhyperelliptic curves happen to be characterized by their canonical models up to isomorphism. We obtain this statement for any curve with at most one non-Gorenstein point, which is unibranch and monomial (Theorem \ref{thmmod}).

As in the case of trigonal curves, F.-O. Schreyer proved in \cite{Sc} that any nonsingular tetragonal curve $C$ can be canonically embedded on a smooth three-dimensional rational normal scroll (a \emph{threefold scroll} for short) $S$ as a complete intersection of two surfaces inside it. Again, the arguments extend naturally to Gorenstein curves and, again, we do here a quite similar reverse engeneering: we assume $C$ is non-Gorenstein, $C'\subset S$ as a local complete intersection, and analyze when or if the main properties on $C$ hold according to the invariant $\ell$ (Theorem \ref{thmtre}). As it happens in the surface scroll case, we also prove that the canonical model of a rational monomial curve lies on a threefold scroll if and only if the curve is at most tetragonal (Theorem \ref{thmsc3}).

\

\noindent{\bf Acknowledgments.} Part of this work can be found in the first named author's Ph. D. thesis \cite{L}. The second named author is partially supported by FAPESP grant numbers 07481-1/2012 and 19676-7/2014. The third named author is partially supported by CNPq grant number 307978/2012-5.


\section{Preliminaries}

For the remainder, a \emph{curve} is an integral and complete one-dimensional scheme over an algebraically closed ground field. So Let $C$ be a curve of (arithmetic) genus $g$ with structure sheaf $\oo_C$, or simply $\oo$. Let $\pi :\overline{C}\rightarrow C$ be the normalization map. Set $\overline{\oo}:=\pi _{*}(\oo _{\overline{C}})$ and call $\ccc:=\mathcal{H}\text{om}(\overline{\oo},\oo)$, the conductor of $\overline{\oo}$ into $\oo$. Let also $\ww_C$, or simply $\ww$, denote the dualizing sheaf of $C$. A point $P\in C$ is said to be \emph{Gorenstein} if $\ww_P$ is a free $\oo_P$-module. The curve is said to be \emph{Gorenstein} if all of its points are so, or equivalently, $\ww$ is invertible. It is said to be \emph{hyperelliptic} if there is a morphism $C\rightarrow\pp^1$ of degree $2$.

\begin{defi} 
\label{defsis}
\emph{A \emph{linear system of dimension $r$ in $C$} is a set of the form
$$
{\rm L}:=\dd(\aaa ,V):=\{x^{-1}\aaa\ |\ x\in V\setminus 0\}
$$
where $\mathcal{F}$ is a coherent fractional ideal sheaf on $C$ and $V$ is a vector subspace of $H^{0}(\aaa )$ of dimension $r+1$. The \emph{degree} of the linear system is the integer
$$
d:=\deg \aaa :=\chi (\aaa )-\chi (\oo)
$$
Note, in particular, that if $\oo\subset\aaa$ then
$$
\deg\aaa=\sum_{P\in C}\dim(\aaa_P/\op).
$$
The notation $g_{d}^{r}$ stands for ``linear system of degree $d$ and dimension $r$". The linear system is said to be \emph{complete} if $V=H^0(\aaa)$, in this case one simply writes ${\rm L}=|\aaa|$. The \emph{gonality} of $C$ is the
smallest $d$ for which there exists a $g_{d}^{1}$ in $C$, or equivalently, a torsion free sheaf $\aaa$ of rank $1$ on $C$ with degree $d$ and $h^0(\aaa)\geq 2$. A point $P\in C$ is called a \emph{base point of ${\rm L}$} if $x\op\subsetneq\aaa_P$ for every $x\in V$. A base point is called \emph{removable} if it is not a base point of $\dd(\oo\langle V\rangle,V)$, where $\oo\langle V\rangle$ is the subsheaf of the constant sheaf of rational functions generated by all sections in $V\subset k(C)$. So $P$ is a non-removable base point of ${\rm L}$ if and only if $\aaa_P$ is not a free $\op$-module; in particular, $P$ is singular if so.}
\end{defi}

Given any integral scheme $A$, any map $\varphi :A\to C$ and a sheaf $\mathcal{G}$ on $C$, set
$$\mathcal{O}_A\mathcal{G}:= \varphi^*\mathcal G/\text{Torsion}(\alpha^*\mathcal G).$$

Given any coherent sheaf $\mathcal{F}$ on $C$ set $\mathcal{F}^n:=\text{Sym}^n\mathcal{F}/\text{Torsion}(\text{Sym}^n\mathcal{F})$. If $\mathcal{F}$ is invertible then clearly $\mathcal{F}^{n}=\mathcal{F}^{\otimes n}$.

\begin{defi}
\label{defcan}
\emph{Call $\widehat{C}:=\text{Proj}(\oplus\,\ww ^n)$ the blowup of $C$ along $\ww$. If $\widehat{\pi} :\widehat{C}\rightarrow C$ is the natural morphism, set $\widehat{\oo}=\widehat{\pi}_*(\oo_{\widehat{C}})$ and $\widehat{\oo}\ww:=\widehat{\pi}_*(\oo _{\widehat{C}}\ww)$. In \cite[p\,188\,top]{R} Rosenlicht showed that the linear system $\sys(\oo_{\overline{C}}\ww,H^0(\ww))$ is base point free. He considered then the morphism $\kappa :\overline{C}\rightarrow\pp^{g-1}$ induced by it and called $C':=\kappa(C)$ the \emph{canonical model} of $C$. He also proved in \cite[Thm\,17]{R} that if $C$ is nonhyperelliptic, the map
$\pi :\overline{C}\rightarrow C$ factors through a map $\pi' : C'\rightarrow C$. So set $\oo':=\pi'_*(\oo_{C'})$ in this case. In \cite[Dfn\,4.9]{KM} one finds another characterization of $C'$. It is the image of the morphism $\widehat{\kappa}:\widehat{C}\rightarrow\pp^{g-1}$ defined by the linear system $\sys(\oo_{\widehat{C}}\ww,H^0(\ww))$. By Rosenlicht's Theorem, since $\ww$ is generated by global sections, we have that $\widehat{\kappa}:\widehat{C}\rightarrow C'$ is an isomorphism if $C$ is nonhyperelliptic.}
\end{defi}

Now set $\overline{\oo}\ww:=\pi_*(\oo_{\overline{C}}\ww)$ and take $\lambda\in H^0(\ww)$ such that $(\overline{\oo}\ww)_P=\overline{\oo}_P\lambda$ for every singular point $P\in C$. Such a differential exists because $H^0(\ww)$ generates $\overline{\oo}\ww$ as proved in \cite[p\,188 top]{R}, and because the singular points of $C$ are of finite number and $k$ is infinite since it is algebraically closed. Set
\begin{equation}
\label{equvvv}
\vv=\vv_{\lambda}:=\ww/\lambda
\end{equation}
If so, we have
$$
\ccc_P\subset
\mathcal{O}_P \subset \vvp\subset\oph=\op'\subset\obp
$$
for every singular point $P\in C$, where the equality makes sense if and only if $C$ is nonhyperelliptic.

\begin{defi} \label{defnng}
\emph{Let $P\in C$ be any point. Set
$$
\eta_P:=\dim(\vvp/\op)\ \ \ \ \ \ \ \ \ \ \ \mu_P:=\dim({\widehat{\oo}_{P}}/\vvp)
$$
and also
$$
\eta:=\sum_{P\in C}\eta_P\ \ \ \ \ \ \ \ \ \ \mu:=\sum_{P\in C}\mu_P
$$
In particular, letting $g'$ be the genus of $C'$, we have
$$
g=g'+\eta+\mu
$$
Following \cite[Prps. 20, 21, 28]{BF}, we call $P$ \emph{Kunz} if $\eta_P=1$ and,
accordingly, we say that $C$ is \emph{Kunz} if all of its non-Gorenstein points are Kunz; we call $P$ \emph{almost Gorenstein} if $\eta_P=1$ and,
accordingly, we say that $C$ is \emph{almost Gorenstein} if all of its points are so. Following \cite[Dfn. 5.7]{KM}, we call $C$ \emph{nearly Gorenstein} if  $\mu=1$, i.e., $C$ is almost Gorenstein with just one non-Gorenstein point. Finally, following \cite[Dfn. 2.15]{KM}, we call $C$ \emph{nearly normal} if $h^0(\oo/\mathcal{C})=1$.}
\end{defi}

 \begin{rem}
\label{remrel}
\emph{The relevance of the concepts above are summarized below:
\begin{itemize}
\item[(i)] $C$ is nearly Gorenstein if and only if it is non-Gorenstein and $C'$ is projectively normal, owing to \cite[Thm. 6.5]{KM}. 
\item[(ii)] $C$  is nearly normal if and only if $C'$ is arithmetically normal, due to \cite[Thm. 5.10]{KM}.
\item[(iii)] $P$ is Gorenstein if and only if $\eta_P=\mu_P=0$, and $P$ is non-Gorenstein if and only if $\eta_P,\mu_P>0$ by \cite[p. 438 top]{BF}. Besides, if $\eta_P=1$ then $\mu_P=1$, by \cite[Prp. 21]{BF}. In particular, a Kunz curve with only one non-Gorenstein point is as close to being Gorenstein as it gets.
\end{itemize}}
\end{rem}

Now we establish few notations on evaluations. Given a unibranch point $P\in C$ and any function $x\in k(C)^*$, set
$$
v_{P}(x):=v_{\pb}(x)\in\zz
$$
where $\pb$ is the point of $\cb$ over $P$. The \emph{semigroup of values} of $P$ is
$$
\sss=\ssp:=v_{P}(\op ).
$$
We also feature two elements of $\sss$, namely:
\begin{equation}
\label{equaab}
\alpha=\alpha_P :={\rm min}(\sss\setminus\{ 0\})\ \ \ \ \text{and}\ \ \ \beta=\beta_P :={\rm min}(v_P(\cp)).
\end{equation}
For later use we set
\begin{equation}
\label{equsss}
\sss^*=\sss_P^*:=\{a\in S\,|\, a\leq\beta\}
\end{equation}
$$
\delta=\delta_P:=\nn\setminus\sss
$$
which agrees with the singularity degree of $P$, that is, $\delta=\dim(\obp/\op)$. The \emph{Frobenius vector} of $\sss$ is $\gamma :=\beta -1$ and one sets
\begin{equation}
\label{equkkp}
\kk=\kk_{P}:=\{ a\in\zz\ |\ \gamma -a\not\in\sss\}
\end{equation}
whose importance will appear later on. We also set, in a bit different way
\begin{equation}
\label{equsss}
\kk^*=\kk_P^*:=\{a\in\sss\,|\, a<\beta\}.
\end{equation}

\begin{defi}
\label{defmon}
\emph{Let $P\in C$. One defines the \emph{multiplicity} of $P$ by
$$
m_C(P)=\dim(\obp/\mmp\obp),
$$
so $P$ is singular if and only if its multiplicity is at least $2$. We say that a unibranch $P$ is \emph{monomial} provided that the completion $\widehat{\op} =k[[t^{n_{1}},\ldots\,,t^{n_r}]]$,
where $t$ is a local parameter at $\pb$.}
\end{defi}

\subsection{Curves on Scrolls}

Following, for instance, \cite{EH,Rd}, a \emph{rational normal scroll} $S:=S_{m_1,\ldots,m_d}\subset\pp^{N}$ with, say, $m_1\leq\ldots\leq m_d$,  is a projective variety of dimensional $d$ which, after a suitable choice of coordinates, is the set of points $(x_0:\ldots: x_N)\subset\mathbb{P}^N$ such that  the rank of
$$
\bigg(
\begin{array}{cccc}
x_0 & x_1 & \ldots & x_{m_1-1} \\
x_1 & x_2 & \ldots & x_{m_1}
\end{array}
\begin{array}{c}
\big{|} \\
\big{|}
\end{array}
\begin{array}{ccc}
x_{m_1+1} & \ldots &  x_{m_1+m_2} \\
x_{m_1+2} & \ldots &  x_{m_1+m_2+1}
\end{array}
\begin{array}{c}
\big{|} \\
\big{|}
\end{array}
\begin{array}{c}
\ldots    \\
\ldots  
\end{array}
\begin{array}{c}
\big{|} \\
\big{|}
\end{array}
\begin{array}{cc}
\ldots & x_{N-1}  \\
\ldots & x_N
\end{array}
\bigg)
$$
is smaller than 2. So, in particular,
\begin{equation}
\label{equnnn}
N=e+d-1
\end{equation}
where $e:=m_1+\ldots+m_d$

Note that $S$ is the disjoint union of $(d-1)$-planes determined by a (parametrized) choice of a point in each of the $d$ rational normal curves of degree $m_d$ lying on complementary spaces on $\mathbb{P}^N$. We will refer to any of these $(d-1)$-planes as a \emph{fiber}. So $S$ is smooth if $m_i>0$ for all $i\in\{1,\ldots,d\}$. From this geometric description one may see that
\begin{equation}
\label{equdgs}
\deg(S)=e
\end{equation}
The scroll $S$ can also naturally be seen as the image of a projective bundle. In fact,  taking $\mathcal{E}:=\oo_{\pum}(m_1)\oplus\ldots\oplus\oo_{\pum}(m_d)$, one has a birational morphism  
\begin{equation*}
 \mathbb{P}(\mathcal{E})\longrightarrow S\subset\mathbb{P}^{N}
\end{equation*}
defined by $\oo_{\mathbb{P}(\mathcal{E})}(1)$.  The morphism is such that any fiber of $\mathbb{P}(\mathcal{E})\to\pum$ is sent to a fiber of $S$. It is an isomorphism if $S$ is smooth. In this case, one may describe the Picard group of the scroll as 
$$
\text{Pic}(S)=\mathbb{Z}H\oplus\mathbb{Z}F
$$
where $F$ is the class of a fiber, and $H$ is the hyperplane class. And one may also compute its Chow ring as
\begin{equation}
\label{equchw}
A(S)=\frac{\mathbb{Z}[H,F]}{(F^2\, ,\, H^{d+1}\, ,\, H^{d}F\, ,\, H^{d}-eH^{d-1}F)}
\end{equation}
From (\ref{equdgs}) we get the relations
\begin{equation}
\label{equrel}
H^d=e\ \ \ \ \ \text{and}\ \ \ \ \ H^{d-1}F=1
\end{equation}
The canonical class in $S$ is given by
\begin{equation}
\label{equccs}
K_S=-dH+(e-2)F
\end{equation}
By \cite[Lem. 3.1, Cor. 3.2]{Mr}, we also have the formulae
\begin{equation}
\label{equhhz}
h^0(\oo_S(aH+bF))=
\begin{cases}
\displaystyle (b+1)\binom{a+d-1}{d-1}+e\binom{a+d-1}{d}
&
{\rm if}\  a\geq 0\ \text{and}\ b\geq-am_1
\\
0 & \text{otherwise} 
\end{cases}
\end{equation}
and
\begin{equation}
\label{equhh1}
h^i(\oo_S(aH+bF))=0\ \ \ \ \ \text{if}\ i\geq 1,\  a\geq 0\ \text{and}\ b\geq -(am_1+1)
\end{equation}
Now let $X$ be a nondegenerate curve lying on $S$. We set the parameter
\begin{equation}
\label{equlll}
\ell:=X\cdot F
\end{equation}
which will be widely studied here. If the curve is a complete intersection inside $S$ written as
$$
X=(a_1H+b_1F)\cdot\ldots\cdot(a_{d-1}H+b_{d-1}F)
$$
then clearly
\begin{equation}
\label{equcll}
\ell=a_1\cdot\ldots\cdot a_{d-1}
\end{equation}
and one may also use (\ref{equchw}) and (\ref{equrel}) to compute $X\cdot H$ and get
\begin{equation}
\label{equcdd}
\deg(X)=a_1\cdot\ldots\cdot a_{d-1}\cdot e+\sum_{i=1}^{d-1}a_1\cdot\ldots\cdot a_{i-1}\cdot b_{i}\cdot a_{i+1}\cdot\ldots\cdot a_{d-1}
\end{equation}


\section{Canonical Models on Surface Scrolls}
\label{sec222}

In this section we study curves for which the canonical model lies on a two-dimensional rational normal scroll, a \emph{surface scroll} for short. So, compared to the previous section, we are in the case where the scroll $S=S_{mn}\subset\mathbb{P}^N$ and $d=2$. If the dimension $N$ of the ambient space is fixed, one may simply write $S=S_{m}$.

If $m>0$ then $S$ is smooth. Let $X=aH+bF\in{\rm Pic}(S)$ be a curve on the scroll. One may use its resolution
$$
0 \longrightarrow \oo_S(-aH-bF) \longrightarrow \oo_S \longrightarrow \oo_X \longrightarrow 0
$$
to compute the arithmetic genus $p_a(X)$ of the curve. In fact,
\begin{align}
\label{equpa2}
p_a(X) &= 1 -\chi(\oo_X)\\
& = 1 -\chi(\oo_S) + \chi(\oo_S(-aH-bF)) = \chi(\oo_S(-aH-bF)) \nonumber
\end{align}
For an arbitrary $D \in \Pic(S)$, by Hirzebruch-Riemann-Roch Theorem we have
\begin{align}
\label{equeul}
\chi(\oo_S(D)) & = \deg\left({\rm ch}(\oo_S(D))({\rm td}(\mathcal{T}_S)\right)_2 \\
& = \frac{1}{12}(c_1^2-c_2)+\frac{1}{2}c_1D+\frac{1}{2}D^2 \nonumber
\end{align}
where $c_i$ is the $i$-th Chern class of the tangent bundle $\mathcal{T}_S$. Now, using  (\ref{equccs}), we have that $c_1=-K_S=2H +(2-e)F$. On the other hand, $\chi(\oo_S)=1$ and, by (\ref{equhhz}) and (\ref{equhh1}), we have that $\chi(\oo_S(H))=e+2$. So, by (\ref{equeul}), and the relations (\ref{equrel}) one may solve a system to get $c_2=4HF$. Writing $D=hH+fF$, and replacing $c_1$ and $c_2$ in (\ref{equeul}) we have
\begin{equation}
\label{equhf2}
\chi(\oo_S(D))= 1+h+f+hf+\frac{he}{2}+\frac{h^2e}{2}
\end{equation}
which is the general formula for a cycle. If the cycle corresponds to a curve, then (\ref{equhf2}) could have been obtained directly from (\ref{equhhz}) and (\ref{equhh1}). In any rate, the genus of the curve derives from (\ref{equpa2}) and (\ref{equhf2}), that is,
$$
p_a(X)= 1-a-b+ab-\frac{ae}{2}+\frac{a^2e}{2}
$$
Now $e=N-1$ by (\ref{equnnn}), $a=\ell$ by (\ref{equcll}), and $b=\deg(X)-\ell(N-1)$ by (\ref{equcdd}). Hence
\begin{equation}
\label{equscr} 
2p_{a}(X)-2=(2\ell-2)\deg(X)-(N-1)\ell^2+(N-3)\ell
\end{equation}
which agrees with \cite[p 66 bot]{SV}, obtained with no explicit use of the Euler characteristic.

If $m=0$, then $S$ is a cone. In this case, \cite[Frm. 3.1]{RS} computes the Hilbert function of the curve to get
\begin{equation}
\label{equcon}
p_{a}(X)=(q-1)(\deg(C)-1)-\frac{1}{2}q(q-1)(N-1)
\end{equation}
where $q=\lceil \deg(X)/(N-1)\rceil$.

It can be read off from \cite{M1, M2, RS, SV} that the canonical model of any trigonal curve lies on a (possibly degenerated) surface scroll. Actually, for Gorenstein curves, this characterizes trigonality. In this section we do precisely the converse, that is, establishing conditions for a non-Gorenstein curve to have its canonical model lying on a surface scroll. We will assume that all curves have genus at least $4$, otherwise the canonical models lie on a plane which is certainly a surface scroll. The results we get are the following.

\begin{teo}
\label{thmtwo}
Let $C$ be a non-Gorenstein curve with $g\geq4$ such that its canonical model $C'$
lies on the scroll $S_{m}$. Consider the standard affine plane chart 
$$
U:=\{(1:x:x^2: \ldots : x^{g-m-2}:y:yx:yx^2:\ldots:yx^m)\,|\, (x,y)\in\mathbb{A}^2\}\cong \mathbb{A}^2
$$
on $S_m$, and assume $C'$ is given by the equation
$$
c_{\ell}(x)y^{\ell}+\ldots+c_{1}(x)y+c_{0}(x)=0\ \ \ \
c_{\ell}(x)\neq 0
$$
on $U$. Then the following hold:
\begin{itemize}
\item[(i)] $\ell\leq3$ if $m>0$ and $\ell\leq 2$ if $m=0$.
\item[(ii)] if $m>0$, then $\ell$ agrees with the generic number of points in the intersection of $C'$ and a member of a pencil of lines on $S_m$. 
\item[(iii)] if $m=0$, then the generic number of points in the intersection of $C'$ and a line on $S_m$ is at most $3$. If equality holds, then $C$ is Kunz with just one non-Gorenstein point;
\item[(iv)] if $\ell=1$ then $C$ is rational, and if $m>0$, then $C'\cong\pum$, in particular, the singular points of $C$ are non-Gorenstein;
\item[(v)] if $\ell=2$ then:
\begin{itemize}
\item[(a)] if $m>0$, then $C$ is nearly Gorenstein;
\item[(b)] if $m=0$, then $g-g'\leq 3$; in particular, $C$ is nearly Gorenstein;
\end{itemize}
\item[(vi)] if $\ell=3$ then:
\begin{itemize}
\item[(a)] $C$ is almost Gorenstein if and only if it is Kunz;
\item[(b)] $m\geq (g-3)/3$ and equality holds if and only if $C$ is Kunz with just one non-Gorenstein point;
\end{itemize}
\item[(vii)] if $m=0$ and $g-g'=3$ then:
\begin{itemize}
\item[(a)] $\cl$ does not meet the vertex if and only if it is hyperelliptic;
\item[(b)] $\cl$ meets the vertex if and only if it is nearly normal;
\item[(c)] ${\rm gon}(C)\leq 5$
\end{itemize}
\item[(viii)] ${\rm gon}(C)\leq {\rm gon}(C')+g-g'$;
\item[(ix)] if $m>0$ then ${\rm gon}(C')\leq\ell$
\end{itemize}
\end{teo}

\begin{proof}
Set $S:=S_m$. We start by proving (ii), i.e., that we are just extending the definition of $\ell$ from smooth scrolls to cones. In fact, by construction, the vertical lines ``$x=c$" on $U$ correspond to fibers of a pencil on $S$ and hence $\ell$ agrees with the one defined in (\ref{equlll}) if $m>0$, so the item follows.  

Let us prove (i). According to \cite[Prp. 2.14]{KM}, the degree of $C'$ in
$\pp ^{g-1}$ can be expressed as
\begin{equation}
\label{equdcl}
\deg(C')=2g-2-\eta
\end{equation}
On the other hand, following \cite{RS,SV}, the degree of $C'$ can also be computed in terms of its equation in the chart $U$, that is, \begin{equation}
\label{grauC} \deg(C')=r+\ell m
\end{equation}
where $r$ is the smallest integer such that
\begin{equation}
\label{graucurv2} \deg c_{i}(x)\leq r-i(g-2-2m)\ \ \ \ \ \ \ i\in\{
0,1,\ldots ,\ell\}
\end{equation}

Assume $m>0$. Applying (\ref{equscr}) to $C'\subset\mathbb{P}^{g-1}$ and using (\ref{equdcl}) we get
\begin{equation}
\label{equmsc}
g'=(\ell-1)\biggl( \biggl( 1-\frac{g}{2} \biggr)
\ell\ +\ (2g-\eta -3)\biggr)
\end{equation}
Note that the second factor at the right hand side of the equality is a decreasing linear function of $\ell$ with root $(4g-2\eta
-6)/(g-2)\leq (4g-8)/(g-2)= 4$. Since $g'\geq 0$, it follows that $\ell\leq 4$. If $\ell=4$ we have $\eta=1$ and $g'=0$, but $\eta =1$ implies $C$ has a unique non-Gorenstein point $P$ with $\eta _{P}=1$ and this yields $\mu _{P}=1$. Therefore $g=g'+\eta+\mu=0+1+1=2$ which cannot happen since we are assuming $g\geq 4$.

Now assume $m=0$. From (\ref{grauC}) and (\ref{graucurv2}) we get $2g-2-\eta-\ell(g-2)\geq 0$, then
\begin{equation}
\label{equgrc}
\ell\leq 1+\frac{g-\eta}{g-2}
\end{equation}
so $\ell\leq 3$ and equality holds if and only if $\eta=1$ and $g=3$, which is off from our assumption on the genus. Therefore (i) is proved.

Let us prove (iii). If $\ell=1$, then a fiber meets $C'$ at most twice, since it meets the curve at most one time in $U$. If $\ell=2$, the fiber can possibly meet $C'$ three times if the latter passes through the vertex. But $C'$ meets the vertex if and only if $\deg c_{2}(x)\geq 1$.  On the other hand, from  (\ref{grauC}) and (\ref{graucurv2}) we have that $\deg c_2(x)\leq 2-\eta$, so this can only happen if $\eta=1$, that is, $C$ is Kunz, with just one non-Gorenstein point.

Now let us suppose $\ell=1$. Then, clearly, $C'$ is rational and so is $C$. If $m>0$, then by (\ref{equmsc}) we have $g'=0$, i.e., $C'\cong\pp ^{1}$. In particular, $C'$ is nonsingular and hence the singular points of $C$ must be non-Gorenstein. So (iv) is proved.

To prove (v), assume $\ell=2$. If $m>0$ we have by (\ref{equmsc}) that $g-g'=\eta+1$, then $\mu=1$ and, by definition, $C$ is nearly Gorenstein, so (a) follows. If $m=0$, by (\ref{equgrc}), we have $\eta\leq 2$. If one applies (\ref{equcon}) to $\cl$ one gets
\begin{equation}
\label{macete}
g'=(q-1)(2g-\eta-3)-\frac{1}{2}q(q-1)(g-2)
\end{equation}
with
\begin{equation}
\label{equqqq}
q=\lceil 1+(g-\eta)/(g-2)\rceil.
\end{equation}
Since $\eta=1$ or $2$, then $q=2$ or $3$. If $q=2$ we have by (\ref{macete}) that $g'-g=\eta+1$ so $g-g'\leq 3$. While if $q=3$ we have by (\ref{equqqq}) that $g=3$ and $\eta=1$. In particular, $g-g'\leq 3$ as well and (b) follows.

To prove (vi), assume $\ell=3$. We have by (\ref{equmsc}) that $g-g'=2\eta$ so clearly $C$ is almost Gorenstein if and only if it is Kunz. Now combining (\ref{grauC}) and (\ref{graucurv2}) for $\ell=3$, one gets
$$
2g-2-\eta-3m-3(g-2-2m)\geq 0
$$
wich yields
$$
m\geq\frac{g-3}{3}+\frac{\eta-1}{3}
$$
so $g\geq (g-3)/3$ and equality holds if and only if $\eta=1$, that is, $g-g'=2$.

Now assume $m=0$ and $g-g'=3$ in order to prove (vii). Then we have that $g-1=g'+2$ and thus $\cl\subset\pp^{g'+2}$. Besides, $g-g'=3$ implies $\eta=2$. Hence, by (\ref{equdcl}),  $\deg(C')=2g'+2$. So $C'$ is a curve of genus $g'$ lying on a cone of $\mathbb{P}^{g'+2}$ with degree $2g'+2$. Therefore $C'$ does not pass through the vertex of the cone if and only if it is hyperelliptic, owing to \cite[Thm 2.1]{S2}, and $\cl$ passes through the vertex if and only if it is nearly normal, owing to \cite[Thm. 2.4]{M1}.

To prove (viii), let $\aal :=\oo _{\cl}\langle 1,x\rangle$ be a torsion free sheaf of rank $1$ which computes the gonality of $C'$, for $x\in k(\cl)=k(C)$. Set $\aaa :=\oo _{C}\langle1,x\rangle$. Since $\pi':C'\rightarrow C$ is birational, it preserves
cohomology by direct images and, in particular, we have that $H^{0}(C,\pi' _{*}(\aal ))=\langle 1,x\rangle$, thus  $\oo
\subset\aaa\subset\pi' _{*}(\aal )$ and also
$$
\deg_{C}\pi' _{*}(\aal)=\deg_{\cl}\aal +g-g'
$$
but ${\rm gon}(C)\leq \deg_C(\aaa)\leq\deg_{C}\pi'_*(\aal)$ and ${\rm gon}(\cl)=\deg_{C'}(\aal)$ so the item follows. If $m>0$ then the ruling of $S$ cuts out a $g_{\ell}^1$ in $C'$ so (vii) follows as well. Now if $m=0$ and $g-g'=3$, then by (vii), ${\rm gon}(C')=2$ in any case, so ${\rm gon}(C)\leq 5$. This proves (ix) and we are done.
\end{proof}


\section{Rational Curves on Surface Scrolls}

We first set up what will be the objects we will deal with in this section. Consider the morphism

\begin{gather*}
\begin{matrix}
\mathbb{P}^1 & \longrightarrow & \mathbb{P}^n \\
   (s:t)              & \longmapsto     & (s^{a_n}:t^{a_1}s^{a_n-a_1}:\ldots:t^{a_{n-1}}s^{a_n-a_{n-1}}:t^{a_n})
\end{matrix}
\end{gather*}
The image $C$ of such a map is what we call a \emph{rational monomial curve}, which, for simplicity, we denote by
$$
C=(1:t^{a_1}:\ldots:t^{a_{n-1}}:t^{a_n})
$$
with $a_1<\ldots < a_n$. A key property for our interest is the result below. In order to state it, recall the definition of $\kk^*_P$ in (\ref{equsss}). 

\begin{prop}
\label{prprat}
Let $C=(1:t^{a_1}:\ldots:t^{a_{n-1}}:t^{a_n})$ be a rational monomial curve such that $a_n=a_{n-1}+1$. Set $P=(1:0:\ldots:0)\in C$. Then its canonical model is
$$
C'=(1:t^{b_1}:\ldots :t^{b_{g-1}})
$$
where $\{0,b_1,\ldots,b_{g-1}\}=\kk_P^*$.
\end{prop}

\begin{proof}
First we just show that $\kk^*$ has $g$ elements. Indeed, since $a_n=a_{n-1}+1$, we have that $P$ is the only singular point of $C$, so $g=\delta=\#(\kk^*)$ by the very construction of $\kk$. Now we will prove that
$$
\vv=\oo_C\langle 1,t^{b_1},\ldots,t^{b_{g-1}}\rangle
$$
where $\vv$ is defined as in (\ref{equvvv}). Set $\pb:=(1:0)\in\pum$ and let $\lambda\in\Omega_{k(C)|k}$ be a differential. We have that $(\ww/\lambda)_P$ is the largest among the fractional $\oo_P$-ideals $F$ on $k(C)$, satisfying the property that ${\rm Res}_{\pb}(f\lambda)=0$ for every $f\in F$. Taking $\lambda:=dt/t^{\beta}$ one sees that
\begin{equation}
\label{equwpl}
(\ww/\lambda)_P=k\oplus kt^{b_1}\oplus\ldots\oplus kt^{b_{g-1}}\oplus\cp
\end{equation}
and hence $\lambda$ is the desired differential introduced above satisfying $(\overline{\oo}\omega)_P=\obp\lambda$, and so $\ww/\lambda=\vv$ since $P$ is the only singular point of $C$. Let $\aaa:=\oo_C\langle 1,t^{b_1},\ldots,t^{b_{g-1}}\rangle$. We claim that $\aaa=\vv$. Since $h^0(\aaa)\geq g$ and $\vv$ is (an isomorphic image of) the dualizing sheaf, it suffices to prove that $\deg(\aaa)=2g-2$. Clearly
$$
\deg\aaa=\dim(\aaa_P/\oo_P)+\dim(\aaa_{\infty}/\oo_{\infty})
$$
where $\infty:=(0:\ldots:0:1)$. Since $\infty$ is smooth, $\dim(\aaa_{\infty}/\oo_{\infty})=b_{g-1}=\beta-2$. On the other hand, $\aaa_P=\vv_P$ and $\dim(\aaa_P/\cp)=g$ by (\ref{equwpl}). But 
\begin{align*}
\beta &=\dim(\obp/\cp) \\
          &=\dim(\obp/\oo_P)+\dim(\aap/\cp)-\dim(\aap/\op) \\
          &=2g-\dim(\aap/\op)
\end{align*}
and the claim follows. Therefore $H^0(\vv)=\langle 1,t^{b_1},\ldots,b^{g-1}\rangle$ and the result trivially follows from the definition of the canonical model.
\end{proof}

Now we prove a result for general scrolls which will also be helpful for our purposes here and in Section 5. We just warn the reader that the statement -- as any envolving monomiality -- depends on a choice of coordinates of the ambient space, though one is allowed at least to reodering them. 

\begin{lema}
\label{pa}
A rational monomial curve $(1:t^{a_1}:\ldots :t^{a_{N}})\subset \pp^{N}$ lies on a $d$-fold scroll $S_{m_1m_2\ldots m_d}$ if and only if there is a partition of the set $\{0=a_0,a_1,\ldots,a_{N}\}$ into $d$ subsets, with, respectively, $m_1+1, m_2+1,\ldots,m_d+1$ elements, such that the elements of all of these subsets can be reordered within an arithmetic progression with the same common difference.
\end{lema}

\begin{proof}
The $d$-fold scroll $S_{m_1m_2\ldots m_d}$ is the set of points $(x_0:\ldots: x_N)\subset\mathbb{P}^N$ such that  the rank of
$$
\bigg(
      \begin{array}{cccc}
        \ \ x_0\, \ldots\, x_{m_1-1} &  x_{m_1+1}\,\ldots \, x_{m_1+m_2} & \ldots & x_{m_1+\ldots+m_{d-1}+d-1}\, \ldots\, x_{N-1}\\
        x_1\,   \ldots\,  x_{m_1}  &  \ \ \ x_{m_1+2}\, \ldots\,  x_{m_1+m_2+1}&\ldots & x_{m_1+\ldots+m_{d-1}+d}\,\ldots\,x_N
\end{array}\bigg)
$$
is smaller than 2.

More explicitly, the above matrix is composed by smaller submatrices of the form
$$
\left(
\begin{array}{cccc}
x_{m_1+\ldots+m_{i}+i}&  x_{m_1+\ldots+m_{i}+i+1}&\ldots & x_{m_1+\ldots+m_{i}+m_{i+1}+i-1}\\
x_{m_1+\ldots+m_{i}+i+1}&  x_{m_1+\ldots+m_{i}+i+2} & \ldots & x_{m_1+\ldots+m_{i+1}+i}
     \end{array}
    \right)$$

Since in our case we have $x_{i}=t^{a_j}$, the submatrices above give a partition of the set of integers $a_i's$, which, sorted conveniently, is
$$
\{a_0,\ldots,a_{m_1}\},\ \{a_{m_1+1},\ldots,a_{m_1+m_2+1}\},\ \ldots\ , \{a_{m_1+\ldots+m_{d-1}+d-1},\ldots,a_N\}.
$$

The elements of each subset of the partition should form an arithmetic progression because of the relations of the submatrix they are related to. The common differences need to be the same due to the relations of the total matrix.
\end{proof}

So Proposition \ref{prprat} and Lemma \ref{pa} provide a way of computing canonical models of rational monomial curves with just one singularity and checking out the scrolls they live within. We exhibit below all non-Gorenstein curves of this form with genus at most $6$ whose canonical models lie on a surface scroll. We just show one scroll (with smallest $m$) for each canonical model, though it can possibly be contained in others and in different ways. Besides, we did not preserve the order on the powers of $t$ and allow them to be negative sometimes after multiplying by $t^s$ for a suitable $s$ so that $C'$ fits into the format of Theorem \ref{thmtwo} and one easily finds its plane equation.  We adopt the conventions: gn:= gonality, K:=Kunz, NG:=nearly Gorenstein, NN:=nearly normal.

\begin{center}
\begin{tabular}{|c|c|c|c|c|c|c|}
 \hline
\multicolumn{7}{|c|}{\textbf{genus 4}}\\
\hline
$C$ &  gn &  & $C'$ & eq & $\ell$ & $m$ \\ 
\hline
$(1:t^5:t^6:t^7:t^8:t^9)$ & 2 & NN & $(1:t:t^2:t^3)$ & $y-x^3$ & 1 & 0  \\
$(1:t^4:t^5:t^7:t^8)$ & 3 & K & $(1:t:t^2:t^{-3})$ & $x^3y-1$ & 1 & 0  \\
$(1:t^4:t^6:t^7:t^8:t^9)$ & 3 & NG & $(1:t^2:t^4:t^3)$ & $y^2-x^3$ & 2 & 0  \\
$(1:t^3:t^7:t^8)$ & 3 & -- & $(1:t:t^3:t^4)$ & $y-x^3$ & 1 & 1  \\ 
\hline
\multicolumn{7}{|c|}{\textbf{genus 5}}\\
\hline
$(1:t^6:t^7:t^8:t^9:t^{10}:t^{11})$ & 2 & NN & $(1:t:t^2:t^3:t^4)$ & $y-x^4$ & 1 & 0  \\
$(1:t^5:t^7:t^8:t^9:t^{10}:t^{11})$ & 3 & NG & $(1:t:t^2:t^3:t^{-2})$ & $x^2y-1$ & 1 & 0  \\ 
$(1:t^5:t^6:t^8:t^9)$ & 3 & NG & $(1:t:t^2:t^3:t^{-3})$ & $x^3y-1$ & 1 & 0  \\
$(1:t^5:t^6:t^7:t^9:t^{10})$ & 3 & K & $(1:t:t^2:t^3:t^{-4})$ & $x^4y-1$ & 1 & 0  \\
$(1:t^4:t^6:t^9:t^{10}:t^{11})$ & 3 & NG & $(1:t^2:t^4:t^6:t^5)$ & $y^2-x^5$ & 2 & 0  \\
$(1:t^4:t^5:t^{10}:t^{11})$ & 3 & -- & $(1:t:t^2:t^{-4}:t^{-3})$ & $x^4y-1$ & 1 & 1  \\
$(1:t^5:t^7:t^8:t^9:t^{10}:t^{11})$ & 3 & NG & $(1:t^2:t^4:t^3:t^5)$ & $y^2-x^3$ & 2 & 1  \\
$(1:t^4:t^7:t^9:t^{10})$ & 3 & -- & $(1:t^3:t^5:t:t^4)$ & $y^3-x$ & 3 & 1  \\
$(1:t^3:t^8:t^9:t^{10})$ & 3 & -- & $(1:t^3:t^6:t^2:t^5)$ & $y^3-x^2$ & 3 & 1  \\
$(1:t^3:t^7:t^{10}:t^{11})$ & 3 & K & $(1:{t}^3:{t^6}:{t^4}:{t^7})$ & $y^3-x^4$ & 3 & 1  \\
\hline
\multicolumn{7}{|c|}{\textbf{genus 6}}\\
\hline
$(1:t^7:t^8:\ldots:t^{12}:t^{13})$ & 2 & NN & $(1:t:t^2:t^3:t^4:t^5)$ & $y-x^5$ & 1 & 0  \\
$(1:t^4:t^6:t^{11}:t^{12}:t^{13})$ & 3 & NG & $(1:t^2:t^4:t^6:t^8:t^7)$ & $y^2-x^7$ & 2 & 0  \\
$(1:t^5:t^8:t^9:t^{11}:t^{12})$ & 3 & -- & $(1:t:t^2:t^3:t^{-3}:t^{-2})$ & $x^3y-1$ & 1 & 1  \\
$(1:t^5:t^6:t^9:t^{12}:t^{13})$ & 3 & -- & $(1:t:t^2:t^3:t^{-4}:t^{-3})$ & $x^4y-1$ & 1 & 1  \\
$(1:t^5:t^6:t^7)$ & 3 & -- & $(1:t:t^2:t^3:t^{-5}:t^{-4})$ & $x^5y-1$ & 1 & 1  \\
$(1:t^5:t^7:t^9:t^{11}:t^{12}:t^{13})$ & 3 & NG & $(1:t^2:t^4:t^6:t^5:t^7)$ & $y^2-x^5$ & 2 & 1  \\
$(1:t^3:t^8:t^{12}:t^{13})$ & 3 & K & $(1:t^3:t^6:t^9:t^5:t^8)$ & $y^3-x^5$ & 3 & 1  \\
$(1:t^4:t^9:t^{10}:t^{11})$ & 3 & -- & $(1:t:t^2:t^4:{t^5}:{t^6})$ & $y-x^4$ & 1 & 2   \\ 
$(1:t^3:t^{10}:t^{11})$ & 3 & -- & $(1:t^3:t^6:t:t^4:t^7)$ & $y^3-x$ & 3 & 2  \\
\hline
\end{tabular}
\end{center}

\

Note that one is able to verify the statements of Theorem \ref{thmtwo}. We have that $\ell\leq 2$ if $m=0$ and $\ell\leq 3$ if $m>0$. The cases where $\ell=2$ are all Kunz, nearly normal or nearly Gorenstein curves (recall that the first ones imply the latter). On the other hand, if $\ell=3$ the curve is either Kunz or of no special kind.

Note also that, at least for genus at most $6$, there are no monomial rational curves lying on a surface scroll with gonality greater than $3$. This is not a particular fact, as we see in the result below.

\begin{teo}\label{scroll2}
The canonical model of a rational monomial curve $C$, with a unique singular point, lies on a surface scroll if and only if $\gon(C)\leq 3$.
\end{teo}

\begin{proof}

The converse is already known. For sufficiency, let
$$
\cl=(1:t^{b_1}:t^{b_2}:\ldots :t^{b_{g-1}})\subset\pp^{g-1}.
$$
and set
$$
A=\{0,b_1,b_2,\ldots,b_{g-1}\}.
$$
Let $\sss$ be the semigroup of the singular point $P\in C$ and $\kk$, defined by means of it. We have
$$
A=\kk^*
$$
and $b_{g-1}=\beta-2=\gamma-1$. Moreover,
$$
\{\gamma-\alpha+1,\ldots,\gamma-1\}\subset A\ \ \text{and}\ \ \gamma-\alpha\not\in A
$$
by the very definition of $\kk$. We have $\gamma-\alpha<0$ if and only if $\alpha=\beta$, that is, $\cp=\mmp$, which is equivalent to saying that $C$ is nearly normal, and so $\gon(C)=2$ owing to \cite[Thm 3.4]{KM} and \cite[Thm 2.1]{M1}. On the other hand, $\gamma-\alpha=0$ never happens by the definition of $\beta$. So we may assume $\gamma-\alpha>0$.

Now, in order to have $\cl$ lying on a surface scroll, by Lemma \ref{pa}, there is a partition of $A$ into two subsets, say $A_1$ and $A_2$ where the first has $0$ as an element, and both form an arithmetic progression with the same common difference. Say $r$ is this common difference.

If $r=1$,  then necessarily
\begin{align*}
A_1&=\{0,1,2,\ldots,a\} \\
A_2&=\{\gamma-\alpha+1,\gamma-\alpha+2,\gamma-\alpha+3,\ldots,\gamma-1\}
\end{align*}
and one can check that
$$
\sss^*=\{0,\alpha,\alpha+1,\alpha+2,\ldots,\gamma-a-1,\beta\}
$$
Now set
$$
\aaa:=\oo_C\langle 1,t\rangle
$$
We have that $\aaa$ has degree 1 at $\infty$ and 0 elsewhere but $P$. Moreover,
$$
\aaa_P=kt\oplus kt^{\gamma-a}\oplus\op
$$
hence $\dim(\aaa_P/\oo_P)=2$ and $\deg(\aaa)=3$. Therefore $C$ is trigonal.

In order to prove the result for $r=2$, we may assume $\alpha\geq 3$, because if $\alpha=1$, then $C=\pum$ and $\gon(C)=1$; if $\alpha=2$ then $C$ is hyperelliptic, in particular, $\gon(C)=2$. In fact, if $\alpha=2$, set $\aaa=\oo_{C}\langle 1,t^2\rangle$. Then $\aaa$ has degree $0$ elsewhere but $\infty$, where it has degree $2$. So $\deg(\aaa)=\gon(C)=2$.

But if $\alpha\geq 3$, then $\gamma-1,\gamma-2\in \kk$. Now one of these two numbers is even, and it has to be in $A_1$ since $r=2$. It follows that all even numbers smaller than $\gamma$ are in $\kk$. Let $b$ be the first odd number in $\kk$.

If $\gamma$ is even, then all positive even numbers smaller than $\gamma$ are not in $\sss$. Hence one sees that
\begin{equation}
\label{equss2}
\sss^*=\{0,\gamma-b+2,\gamma-b+4,\gamma-b+6,\ldots,\beta-2,\beta\}
\end{equation}
and if $\gamma$ is odd, all odd numbers smaller than $\gamma$ are not in $\sss$ and we have that $\sss^*$ has to be as in (\ref{equss2}) again.

Now $\aaa:=\oo_{C}\langle 1,t^2\rangle$, then $\aaa$ has degree $0$ elsewhere but at $P$ and $\infty$, where it has degree $2$. Besides
$$
\aaa_P=kt^2\oplus \op
$$
so $\aaa$ has degree $1$ at $P$ and $\deg(C)=\gon(C)=3$.

To prove the result for $r\geq 3$, we may assume $\alpha\geq 4$, because if $\alpha=3$, then $\oo_{C}\langle 1,t^3\rangle$, has degree $0$ elsewhere but at $\infty$, where it has degree $3$. Therefore we have $\gon(C)\leq\deg(\aaa)=3$.

But if $\alpha\geq 4$, then $\gamma-1,~\gamma-2,~\gamma-3~\in~\kk$. Since $r\geq 3$, these three numbers have to be in different sets of the partition which defines the scroll. This is impossible unless its dimension is at least $3$, which is not the case here.
\end{proof}

We close this section with a simple, though possibly important result, if one is interested on classifying integral curves by means of canonical models.

\begin{teo}
\label{thmmod}
Two nonhyperelliptic curves with just one singular point which is unibranch and monomial are isomorphic if and only if their canonical models are the same.
\end{teo}

\begin{proof}
Necessity is immediate. To prove sufficiency, let $C_1$ and $C_2$ be such curves. Assume $C_1'=C_2'$, i.e., they have the same canonical model. By \cite[Thm. 17]{R}, any nonhyperelliptic curve is birationally equivalent to its canonical model. So the nonsigular models $\cb_1$ and $\cb_2$ are isomorphic. Hence we just have to prove that $\oo_{C_1,P_1}\cong\oo_{C_2,P_2}$ where $P_1$ and $P_2$ are the singular points of $C_1$ and $C_2$, respectively. But the local rings of two monomial points with the same number of branches are isomorphic if and only if their semigroup of values coincide. So let us prove that $\sss_{P_1}=\sss_{P_2}$. We have that $C_1'=C_2'$ if and only if $H^0(\ww_{C_1})=H^0(\ww_{C_2})$ which is equivalent to saying that $H^0(\vv_{C_1})=H^0(\vv_{C_2})$.  Now, for $i=1,2$, one may write $H^0(\vv_{C_i})=\langle x_1,\ldots,x_{\overline{g}_i},y_1,\ldots,y_{\delta_{P_i}}\rangle$, where $\overline{g}_i$ is the genus of $\overline{C}_i$ , and $\delta_{P_i}$ is the singularity degree of $P_i$.  Since $\overline{g}_1=\overline{g}_2$ it follows that $\delta_{P_1}=\delta_{P_2}$. Using the fact that the points are unibranch,  we thus have $v_{P_i}(\{ y_1,\ldots,y_{\delta_{P_i}}\})=\text{K}_{P_i}^*$ for $i=1,2$. Therefore $\text{K}_{P_1}^*=\text{K}_{P_2}^*$. So it suffices to show that the semigroups have the same conductor. In fact, if so, $\kk_{P_1}=\kk_{P_2}$.  But if the  latter equality holds, take  $s\in \sss_{P_1}$. Then $\gamma-s\notin \kk_{P_1}=\kk_{P_2}$. So $s=\gamma-(\gamma-s)\in \sss_{P_2}$. Analogously we have that $\sss_{P_2}\subset \sss_{P_1}$. Thus $\sss_{P_1}=\sss_{P_2}$ as we wish. 

So let us prove that the semigroups have the same conductor, i.e., $\beta_{P_1}=\beta_{P_2}$. We claim that $\gamma_{P_1}<\beta_{P_2}$. In fact, assume $\gamma_{P_1}=\beta_{P_2}$, then $\gamma_{P_2}=\gamma_{P_1}-1\in \kk_{P_1}^*=\kk_{P_2}^*$,  so $\gamma_{P_2}\in \kk_{P_2}^*$ which cannot happen. Now, if $\gamma_{P_1}>\beta_2$, we have $\gamma_{P_1}-1\in \kk_{P_1}^*=\kk_{P_2}^*$, but $\gamma_{P_1}-1\geq\beta_{P_2}$ and hence cannot be in $\kk_{P_2}^*$ which is another contradiction. So the claim follows and, analogously, $\gamma_{P_2}<\beta_{P_1}$ and therefore $\gamma_{P_1}<\beta_{P_2}$ and $\gamma_{P_1}\notin \kk_{P_2}^*$ implies $\gamma_{P_1}\notin \kk_{P_2}$ and, similarly, $\gamma_{P_2}\notin \kk_{P_1}$. So, by definition, $\gamma_{P_2}-\gamma_{P_1}\in \sss_{P_2}$ and $\gamma_{P_1}-\gamma_{P_2}\in \sss_{P_1}$, but $\sss_{P_1},\sss_{P_2}\subset\nn$ and then $\gamma_{P_2}-\gamma_{P_1}\geq0$ and $\gamma_{P_1}-\gamma_{P_2}\geq0$. Thus $\gamma_{P_1}=\gamma_{P_2}$ and $\beta_{P_1}=\beta_{P_2}$, as desired.
\end{proof}

Although we did not prove any systematic way of obtaining canonical models of rational monomial curves with two singularities -- as we did in Proposition \ref{prprat} in the case of a unique one -- if the genus is low, one can test by hand a candidate for an immersion of the dualizing sheaf into the constant sheaf of rational functions  by checking the properties that it has to have $g$ linearly independent global sections and degree $2g-2$. With this strategy we exhibit the full list of rational monomial curves of genus at most $5$, with two singular points, whose canonical models lie on a surface scroll. We adopt the conventions: $0:=(1:0:\ldots:0)$, $\infty:=(0:\ldots:0:1)$ and, as before, $\delta$ is the singularity degree.

\

\begin{center}
\begin{tabular}{|c|c|c|c|}
\hline
\multicolumn{4}{|c|}{\textbf{genus 4 with $\delta_0=1$ and $\delta_\infty=3$}}\\
\hline
$C$ & gn & $C'$ & $m$ \\
\hline
$(1:t^2:t^3:t^4:t^5:t^9)$ & 3 & $(1:t^2:t^3:t^4)$ &  0\\
$(1:t^2:t^3:t^4:t^6:t^9)$ & 3 & $(1:t^2:t^3:t^5)$ & 1\\

\hline
\multicolumn{4}{|c|}{\textbf{genus 4 with $\delta_0=2$ and $\delta_\infty=2$}}\\
\hline
$(1:t^3:t^4:t^5:t^7:t^9)$ & 3 & $(1:t:t^3:t^5)$ &  0\\
$(1:t^3:t^4:t^5:t^8)$   & 3     & $(1:t:t^3:t^4)$  & $1$\\
\hline
\multicolumn{4}{|c|}{\textbf{genus 5 with $\delta_0=1$ e $\delta_\infty=4$}}\\
\hline
$(1:t^2:t^3:t^4:t^5:t^9)$      & 3               & $(1:t^2:t^3:t^4:t^5)$   & 0 \\
$(1:t^2:t^3:t^5:t^9)$      & 3               & $(1:t^2:t^3:t^4:t^6)$   & 0 \\
$(1:t^2:t^3:t^7:t^{10})$      & 3               & $(1:t^2:t^3:t^5:t^6)$   & 1 \\
\hline
\multicolumn{4}{|c|}{\textbf{genus 5 with $\delta_0=2$ e $\delta_\infty=3$}}\\
\hline
$(1:t^3:t^4:t^5:t^7:t^{9}:t^{11})$     & 3          & $(1:t:t^3:t^5:t^7)$   & 0 \\
$(1:t^2:t^4:t^5:t^6:t^7:t^{11})$  & 3 & $(1:t^2:t^4:t^5:t^6)$  & 0\\
$(1:t^3:t^4:t^5:t^6:t^{10})$      & 3               & $(1:t:t^3:t^4:t^5)$   & 1 \\
$(1:t^3:t^4:t^5:t^7:t^{10})$  & 3 & $(1:t:t^3:t^4:t^6)$  & $1$\\
$(1:t^2:t^5:t^7:t^9:t^{12})$ & 3   & $(1:t^2:t^4:t^5:t^7)$   & 1\\
\hline
\end{tabular}\end{center}

\

Note that all of them have gonality $3$ as well.


\section{Canonical Models on Threefold Scrolls}

In this section, we study curves for which the canonical model lies on a three-dimensional rational normal scroll, a \emph{threefold scroll} for short. The aim is obtaining similar results as the ones we got in Section 2. But here we assume the canonical model is a local complete intersection. As we did before, we start by obtaining a general formula envolving the main invariants of any such a curve and, afterwards, we apply it to the case where it is a canonical model.

So let $X$ be a local complete intersection curve which lies on a threefold scroll $S\subset\mathbb{P}^N$. Applying the Hartshorne-Serre correspondence, see for example \cite{A}, one can find a uniquely determined rank 2 vector bundle $E$ on $S$, such that $X$ is the zero locus of a global section of $E$. We then get a short exact sequence of the form
$$
0 \longrightarrow \oo_S \longrightarrow E \longrightarrow \mathcal{I}_C \otimes \bigwedge^2 E \longrightarrow 0
$$
which induces the following resolution of the structure sheaf of $X$
\begin{equation}
\label{equrs3}
0 \longrightarrow \bigwedge^2 E^\lor \longrightarrow E^\lor \longrightarrow \oo_S \longrightarrow \oo_X \longrightarrow 0.
\end{equation}
Writing
$$
c_1(E)=uH+vF
$$
we will say that $X$ is of $(u,v)$-\emph{type} 

Now the fundamental class of $X$ is given by the second Chern class of $E$, say,
$$
X=c_2(E)=wH^2 + zHF
$$
From (\ref{equchw}) and (\ref{equrel}) we get $(wH^2 + zHF).F=w$, whereas by (\ref{equlll}) we are led to
\begin{equation}
\label{equwel}
w=\ell
\end{equation} 
the instersection number of $X$ and a general fiber of $S$. Similarly, form (\ref{equrel}) we get $(wH^2 + zHF).H = we + z$, and by (\ref{equnnn}) and (\ref{equwel}) we have
\begin{equation}
\label{equzzz}
z=\deg(X)-\ell(N-2)
\end{equation}
In order to compute the genus of $X$ we use the resolution (\ref{equrs3}) which yields to
\begin{equation}
\label{equpax}
p_a(X)=\chi(E^\lor) - \chi\big(\bigwedge^2 E^\lor\big)
\end{equation}
Again, we will firstly use Hirzebruch-Riemann-Roch Theorem to compute these Euler characteristics. To begin with, we start with an arbitrary $D\in{\rm Pic}(S)$. Let $\mathcal{T}_S$ be the tangent bundle and write $c_i:=c_i(\mathcal{T}_S)$ for its Chern classes. We have
$$
\chi(\oo_S(D))=\frac{1}{24}\,c_1c_2+\frac{1}{12}\,D(c_1^2+c_2)+\frac{1}{4}\,D^2c_1+\frac{1}{6}\,D^3                      
$$
Now  (\ref{equccs}) yields
$$
c_1=-K_S=3H+(2-e)F
$$
while one can compute and obtain 
\begin{equation}
\label{equc2T}
c_2=3H^2+(6-2e)HF
\end{equation}
Writing $D=hH+fF$, one gets
\begin{equation}
\label{equxhf}
\chi(\oo_S(hH+fF))=1+\frac{2e+9}{6}h+f+\frac{e+1}{2}h^2+\frac{3}{2}hf+\frac{e}{6}h^3+\frac{1}{2}h^2f
\end{equation}
which allows us quickly calculating the first summand in the right hand side of (\ref{equpax}), that is,
\begin{align}
\label{equced}
\chi(E^\lor)&=2-\frac{(2e+9)}{12}u-v-(e+1)w-\frac{3}{2}z+\frac{(e+1)}{2}u^2+\frac{3}{2}uv\\
                   &\ \ \ \ \ \ \ \ \ \ \ \ \ \ \ \ \ \ \ \ \ \ \ \ \ \ \ \ \ \ \ \ \ \ \ \ \ \ \ \ +\frac{e}{2}uw+\frac{1}{2}uz+\frac{1}{2}vw-\frac{e}{6}u^3-\frac{1}{2}u^2v\nonumber
\end{align}
and immediately the second one since $\chi(\wedge^2 E^\lor)=\chi(\oo_S(-uH-vF))$. Thus
\begin{equation}
\label{equpas}
p_a(X) = 1 +\frac{(e(u-2)+v-2)w+(u-3)z}{2}
\end{equation}
Using (\ref{equnnn}), (\ref{equwel}) and (\ref{equzzz}), we are led to
\begin{equation}
\label{equgne}
2p_a(X)-2=(u-3)\deg(X)+\ell(v+N-4)
\end{equation}
If $E$ splits, i.e, $E=\oo_S(aH+bF)\oplus\oo_S(cH+dF)$, that is, $X=(aH+bF)(cH+dF)$ is a complete intersection inside $S$, then we have the relations
$$
\begin{array}{cc}
u=a+c\ \ \ \ \ & \ \ w=ac\\
v=b+d\ \ \ \ \ & \ \ \ \ \ \ \ \ z=ad+bc
\end{array}
$$
So we could have used the resolution
$$
0\rightarrow\oo_S(-(a+c)H-(b+d)F)\to\oo_S(-aH-bF)\oplus\oo_S(-cH-dF)\to\oo_S\to\oo_X\to 0
$$
to get the following formula
\begin{equation}
\label{equgnn}
2p_a(X)-2=(a+c-3)\deg(X)+\ell(b+d+N-4)
\end{equation}
of which (\ref{equgne}) is a generalization.

We just note  we could had applied Hirzebruch-Riemann-Roch Theorem to write
$$
\chi(E^\lor) =1+\chi\big(\bigwedge^2 E^\lor\big) +\frac{1}{2}K_Sc_2(E)+\frac{1}{2}c_1(E)c_2(E)
$$
and then
$$
p_a(X)=1+\frac{c_2(E)(K_S+c_1(E))}{2}
$$
which provides a faster way of getting (\ref{equgne}) (and hence (\ref{equgnn})) with no need of (\ref{equc2T}), (\ref{equxhf}) and (\ref{equced}) either. We opted to keep these equations since they stand for helpful formulas expressing the Euler characteristics of rank $1$ and $2$ bundles on $S$.

Now we analyze curves for which the canonical model lies on a smooth threefold scroll $S_{mnk}$. Since the scroll is nonsingular, $m,n,k\geq 1$, thus, by (\ref{equnnn}), the ambient dimension is at least $5$, and hence the genus of the curve is greater or equal to $6$. As in the surface scroll case, we will set $S_m:=S_{mnk}$ ($m\leq n\leq k$) because the parameter $m$ will be analyzed, though $n,k$ are not determined by $m$, but just the sum of them. 

\begin{teo}
\label{thmtre}
 Let $C$ be a non-Gorenstein curve of genus $g\geq 6$, whose canonical model $C'$ lies on a smooth threefold scroll $S_m$ as a local complete intersection of $(u,v)$-type. Let $\ell$ be the generic number of points in the intersection of $C'$ and a member of a pencil of planes on $S_m$. Then the following hold:
\begin{itemize}
\item[(i)] $v=-(g-5)$ if and only if $\ell=2$ and $C'$ is elliptic. Otherwise
$$
\ell=\frac{(u-4)(2g-2-\eta)+\eta+2\mu}{5-g-v}
$$
\item[(ii)] if $\ell=1$ then $C'\cong\pum$ and, in particular, $C$ is rational and the singular points of $C$ are non-Gorenstein. Moreover, $C$ is nearly Gorenstein if and only if $v=2$.
\item[(iii)] if $\ell=2$ and $C'$ is not elliptic, then either $C'\cong\pum$ with $v=-(g-4)$ or else $v>-(g-5)$. Moreover, $C$ is nearly Gorenstein if and only if $v=3-\eta$; in particular, it is also Kunz if and only if $v=2$.
\item[(iv)] if $\ell=3$, then $v=-(g-5)-((\eta+2\mu)/3)$; in particular, $C$ is Kunz with just one non-Gorenstein point if and only if $v=-(g-4)$.
\item[(v)] if $\ell=4$ then $v=-(g-5)-((\eta+2\mu)/4)$ or $v=-(g-5)-((g+\mu+1)/2)$; in the first case holds: $C$ cannot be Kunz with just one non-Gorenstein point, $m\geq (4(g-5)+\eta+2\mu)/16$, and it generalizes the Gorenstein case.
\item[(vi)] if $\ell\geq 5$ then
$$
m\geq \frac{\ell(g-5)+(\sqrt{2\ell}-4)(2g-2-\eta)+\eta +2\mu}{\ell(\ell+1)}
$$
\end{itemize}
\end{teo}

\begin{proof}
Since the scroll is nonsingular and the planes on a pencil of $S_m$ are precisely the fibers, then $\ell$ agrees with the intersection number $C'\cdot F$ priorly defined. Taking $X=C'$, $p_a(X)=g-\eta-\mu$, $\deg(X)=2g-2-\eta$, and $N=g-1$ in (\ref{equgne}) we have
\begin{equation}
\label{equgen}
(u-4)(2g-2-\eta)+\eta+2\mu+\ell(v+g-5)=0
\end{equation}
If $v\neq-(g-5)$, the above equality provides the formula of item (i). 

For the remainder, in order to obtain conditions on $u$ and $v$, we will assume, without loss in generality, that the canonical model is a complete intersection given by $C'=(aH+bF)(cH+dF)$ with $u=a+c$ and $v=b+d$. Since the cycles need to be effective, from (\ref{equhhz}) we have that $a,c\geq 0$; besides, $C'$ is non-degenerate, hence $a,c>0$ and also if $b<0$ (resp. $d<0$) then $a\geq 2$ (resp. $c\geq 2$).

Now assume $v=-(g-5)$, then from (\ref{equgen}) we get
$$
\eta+2\mu=(4-u)(2g-2-\eta)
$$
If $\eta=\mu=0$ we are in the classical (Gorenstein) case, i.e., the canonical curve is given  by 
$C=C'=(2H+bF)(2H+dF)$ with $u=4$ and $v=b+d=-(g-5)$. Otherwise, $u<4$ since the left hand side of the equation is strictly positive. But from what was said above, $u\geq 2$ and if equality holds, then $a=c=1$, which implies both $b$ and $d$ are positive, and so would be $v$, which is not the case here. Therefore $u=3$, thus $\ell=ac=2$ by (\ref{equcll}) and 
$$
g=1+\eta+\mu
$$
that is, $C'$ is of genus $g'=1$, so necessity is proved in the claim of (i). We leave sufficiency for the analisys of the case where $\ell=2$ right away.

Now assume $\ell=1$. First off, the planes on $S$ cut out a $g_1^1$ on $C'$, so $C'\cong\pum$ and, in particular, $C$ is rational, and its singular points are non-Gorenstein. Besides, since $g'=0$ one may write $g=\eta+\mu$, and since $\ell=1$ one should take $u=2$. Replacing this in (\ref{equgen}), we get
$$
v-\mu-1=0
$$
and item (ii) follows.

If $\ell=2$ then $u=3$ and (\ref{equgen}) yields
$$
\eta +\mu+v-4=0
$$
So one may write $v=-(g-g'-5)$. If $g'=0$ then $C'\cong\pum$ with $v=-(g-4)$; if $g'=1$, i.e., $C'$ is elliptic, then we have sufficiency in the claim of item (i). Otherwise, if $g'\geq 2$, then $v>-(g-5)$. Now, from (\ref{equgen}), $C$ is nearly Gorenstein if and only if $v=3-\eta$ and also Kunz if and only if $v=2$.

If $\ell=3$ then $u=4$ and so
$$
\eta +2\mu+3(v+g-5)=0
$$
which immediately implies the equation (iv). The last sentence of the item follows from the fact that $C$ is Kunz with just one non-Gorenstein point if and only if $\mu=\eta=1$.

If $\ell=4$, either $u=4$ or else $u=5$. The first case occurs when $a=c=2$ so it generalizes the canonical Gorenstein case of tetragonal curves, as mentioned above. The equation (\ref{equgen}) gives the relation
$$
\eta +2\mu+4(v+g-5)=0
$$
from where the value of $v$ comes. It does not allow $\eta=\mu=1$ otherwise $v$ would not be integer, so $C$ cannot be Kunz with one non-Gorenstein point. From (\ref{equhhz}), we deduce that $v=(b+d)\geq -m(a+c)=-mu$ and the lower bound for $m$ can  easily be found. The second case provides the relation
$$
2g-2+2\mu+4(v+g-5)=0
$$ 
which gives $v$.

If $\ell\geq 5$, then $v$ is always computed the same way by (\ref{equgen}), that is,
$$
v= -(g-5)-\frac{(u-4)(2g-2-\eta)+\eta +2\mu}{\ell}
$$
where none of the terms vanish. The lower bound for $m$ immediately comes from the fact that $v\geq -mu$, $u=a+c$, $\ell=ac$ and the general relation for positive integers $\sqrt{2ac}\leq a+c\leq ac+1$.
\end{proof}


\section{Rational Curves on Threefold Scrolls}

In this section we study rational monomial curves with canonical models lying on a threefold scroll. We start by listing all non-Gorenstein curves of this type with genus at most $8$,  with a unique singular point, for which the canonical model is not contained on a surface scroll.

\begin{center}
\begin{tabular}{|c|c|c|c|}
\hline
\multicolumn{4}{|c|}{\textbf{genus 6}}\\
\hline
$C$ &  gn &   $C'$  & $mn$\\ 
\hline
$(1:t^5:t^6:t^{13}:t^{14})$       & 4 &$(1:t^2:t^5:t^6:t^7:t^8)$ & 00\\
$(1: t^4:t^7:t^8:t^9)$                & 4 & $(1:t^4:t^5:t^7:t^8:t^9)$ & $01$\\
$(1:t^4:t^7:t^{10}:t^{12}:t^{13})$       & 4 &$(1:t^3:t^4:t^6:t^7:t^8)$ & $01$\\
$(1:t^5:t^7:t^{8}:t^{11}:t^{12}:t^{13})$& 4 &$(1:t^3:t^5:t^6:t^7:t^8)$ & 01\\ 
\hline
\multicolumn{4}{|c|}{\textbf{genus 7}}\\
\hline
$(1:t^4:t^{10}:t^{11}: t^{12}:t^{13})$    & 4 &$(1:t^2:t^3:t^4:t^6:t^7:t^8)$    &   00 \\
$(1:t^5:t^7:t^{11}:t^{12}:t^{13})$       & 4 &$(1:t:t^3:t^5:t^6:t^7:t^8)$   &   01    \\
$(1:t^5:t^8:t^{11}:t^{12}:t^{13}:t^{14})$& 4 &$(1:t^2:t^3:t^5:t^6:t^7:t^8)$  &     01 \\
$(1:t^6:t^8:t^{9}:t^{11}:t^{12}:t^{13})$ & 4 &$(1:t^3:t^5:t^6:t^7:t^8:t^{9})$   &  01 \\ 
$(1:t^5:t^7:t^{9}:t^{13}:t^{14})$        & 4 &$(1:t^3:t^5:t^7:t^8:t^9:t^{10})$    & 01\\
$(1:t^5:t^8:t^{9}:t^{11}:t^{12})$        & 4 &$(1:t^4:t^5:t^7:t^8:t^9:t^{10})$  &  01 \\
$(1:t^6:t^8:t^{9}:t^{10}:t^{12}:t^{13})$ & 4 &$(1:t^4:t^6:t^7:t^8:t^9:t^{10})$ & 01   \\
$(1:t^5:t^8:t^{9}:t^{10}:t^{11})$        & 4 &$(1:t^5:t^6:t^8:t^9:t^{10}:t^{11})$ & 01 \\
$(1:t^4:t^7:t^{12}:t^{13})$              & 4 &$(1:t:t^4:t^5:t^7:t^8:t^9)$  &   02    \\
$(1:t^4:t^9:t^{14}:t^{15})$        & 4 &$(1:t^3:t^4:t^5:t^7:t^8:t^9)$     &  02 \\
$(1:t^4:t^9:t^{10}:t^{11}:t^{15})$       & 4 &$(1:t^4:t^5:t^6:t^8:t^9:t^{10})$   &  02\\
$(1:t^5:t^7:t^8)$                 & 4 &$(1:t^2:t^5:t^7:t^8:t^9:t^{10})$   &  11\\
$(1:t^6:t^7:t^{8}:t^{9}:t^{10})$        & 4 &$(1:t^2:t^6:t^7:t^8:t^9:t^{10})$    & 11\\
\hline
\multicolumn{4}{|c|}{\textbf{genus 8}}\\
\hline
$(1:t^4:t^{10}:t^{13}:t^{14}:t^{15})$      & 4 &$(1:t^2:t^4:t^5:t^6:t^8:t^{9}:t^{10})$  & 00   \\
$(1:t^6:t^{8}:t^{11}:t^{13}:t^{14}:t^{15})$       & 4 &$(1:t:t^3:t^5:t^6:t^7:t^8:t^9)              $ & 01\\
$(1:t^6:t^{8}:t^{9}:t^{12}:t^{13})$               & 4 &$(1:t:t^4:t^6:t^7:t^8:t^9:t^{10})$ & 01\\
$(1:t^6:t^{9}:t^{10}:t^{13}:t^{14}:t^{16}:t^{17})$& 4 &$(1:t^3:t^4:t^6:t^7:t^{8}:t^{9}:t^{10})$  & 01\\
$(1:t^6:t^{8}:t^{10}:t^{11}:t^{14}:t^{15})$       & 4 &$(1:t^4:t^6:t^8:t^9:t^{10}:t^{11}:t^{12})$ & 01\\
$(1:t^6:t^{9}:t^{10}:t^{11}:t^{13}:t^{14})$       & 4 &$(1:t^5:t^6:t^8:t^9:t^{10}:t^{11}:t^{12})$  & 01\\
$(1:t^6:t^{9}:t^{10}:t^{11}:t^{12}:t^{13})$       & 4 &$(1:t^6:t^7:t^9:t^{10}:t^{11}:t^{12}:t^{13})$  & 01\\
$(1:t^6:t^{9}:t^{11}:t^{14}:t^{15}:t^{16})$& 4 &$(1:t^2:t^3:t^5:t^6:t^{7}:t^{8}:t^{9})$  & 02\\
$(1:t^5:t^9:t^{12}:t^{13}:t^{15}:t^{16})$  & 4 &$(1:t^3:t^4:t^5:t^7:t^{8}:t^{9}:t^{10})$    & 02\\
$(1:t^4:t^{7}:t^{16}:t^{17})$              & 4 &$(1:t^3:t^4:t^7:t^8:t^{10}:t^{11}:t^{12})$  & 02\\
$(1:t^5:t^{8}:t^{14}:t^{16}:t^{17})$       & 4 &$(1:t^3:t^5:t^6:t^8:t^{9}:t^{10}:t^{11})$  & 02\\
$(1:t^5:t^{9}:t^{11}:t^{13}:t^{16}:t^{17})$& 4 &$(1:t^4:t^5:t^6:t^8:t^{9}:t^{10}:t^{11})$  & 02\\
$(1:t^5:t^{6}:t^{13}:t^{14})$              & 4 &$(1:t^4:t^5:t^6:t^9:t^{10}:t^{11}:t^{12})$ & 02\\
$(1:t^5:t^{9}:t^{11}:t^{12})$       & 4 &$(1:t^5:t^6:t^7:t^9:t^{10}:t^{11}:t^{12})$  & 02\\
$(1:t^5:t^{6}:t^{12}:t^{13})$              & 4 &$(1:t^5:t^6:t^7:t^{10}:t^{11}:t^{12}:t^{13})$ & 02\\
$(1:t^4:t^{9}:t^{14}:t^{15})$       & 4 &$(1:t:t^4:t^5:t^6:t^8:t^{9}:t^{10})$     &   11\\
$(1:t^5:t^{7}:t^{13}:t^{15}:t^{16})$       & 4 &$(1:t^2:t^3:t^5:t^7:t^{8}:t^{9}:t^{10})$   &   11 \\
$(1:t^5:t^{7}:t^{8}:t^{9})$               & 4 &$(1:t^2:t^5:t^7:t^9:t^{10}:t^{11}:t^{12})$   &  11\\
$(1:t^4:t^{10}:t^{11}:t^{16}:t^{17})$      & 4 &$(1:t^4:t^6:t^7:t^8:t^{10}:t^{11}:t^{12})$ & 11\\
$(1:t^4:t^{9}:t^{10}:t^{11})$              & 4 &$(1:t^4:t^7:t^8:t^9:t^{11}:t^{12}:t^{13})$ & 11 \\
$(1:t^4:t^{11}:t^{13}:t^{14})$      & 4 &$(1:t:t^3:t^4:t^5:t^7:t^8:t^{9})$     &    12  \\
$(1:t^5:t^{8}:t^{12}:t^{13}:t^{14})$       & 4 &$(1:t^2:t^4:t^5:t^7:t^8:t^{9}:t^{10})$    &  12\\ 
\hline
\end{tabular}
\end{center}

Note that the gonality of each such a curve is $4$, which, as in Section 3, is not a particular fact as one can see from the statement of the result below.

\begin{teo}
\label{thmsc3}
The canonical model of a rational monomial curve $C$, with a unique singular point, lies on a $3$-fold scroll but not on surface scroll if and only if $gon(C)=4$.
\end{teo}

\begin{proof} Let $\cl$, $\sss$, $\kk$ and $A$ be as in the proof of Theorem \ref{scroll2}. Assume first that $\cl$ lies on a $3$-fold scroll but not on a surface scroll. By Lemma \ref{pa}, there is a partition of $A$ into three subsets, say $A_1$, $A_2$ and $A_3$ where the first has $0$ as an element and the third has $\gamma-1$, and all three form an arithmetic progression with the same common difference. Denote by $r$ is this common difference. If $r=1$,  then necessarily
\begin{align*}
A_1=&\{0,1,2,\ldots,a\}\\
A_2=&\{b,b+1,b+2,\ldots,c\} \\
A_3=&\{\gamma-\alpha+1,\gamma-\alpha+2,\gamma-\alpha+3,\ldots,\gamma-1\}
\end{align*}
and one can check that
$$
\sss=\{0,\alpha,\alpha+1,\alpha+2\ldots,\gamma-c-1,\gamma-b+1,\gamma-b+2,\gamma-b+3,\ldots,\gamma-a-1,\beta \rightarrow\}
$$
Now set
$$
\aaa:=\oo_C\langle 1,t\rangle
$$
We have that $\aaa$ has degree 1 at $\infty$ and 0 elsewhere but $P$. Moreover,
$$
\aaa_P=kt\oplus kt^{\gamma-c}\oplus kt^{\gamma-a}\oplus\op
$$
hence $\dim(\aaa_P/\oo_P)=3$ and $\deg(\aaa)=4$. Hence $\gon(C)\leq 4$. It has to be $4$, because otherwise $C'$ would lie on a surface scroll by Theorem \ref{scroll2}.

If $r=2$, write
\begin{align*}
A_1&=\{0,2,4,\ldots,a\}\\
A_2&=\{b,b+2,b+4,\ldots,c\}\\
A_3&=\{d,d+2,d+4\ldots,e\}
\end{align*}
with $b<d$. We may assume $\alpha\geq3$. But if $\alpha\geq 3$, then $\gamma-1, \gamma-2 \in\kk$. We will consider two cases.

\noindent{\bf Case 1.} $\gamma$ is even. \\
If so, $e=\gamma-1$. If $a=\gamma-2$, then all even numbers are in $\kk$, $b$ and $d$ are odd, and, also, all positive even numbers smaller than $\gamma$ are not in $\sss$ since the symmetric element $\gamma-i$ is even for any even integer $i$; so $\alpha$ is odd. Besides $\alpha>\gamma/2$ since $2\alpha$ is even, and no positive even numbers are in $\sss^*$, so $2\alpha>\beta$. Moreover, $\gamma-\alpha$ is odd and is not in $\kk$, so $d=\gamma-\alpha+2$ and $c<\gamma-\alpha$. One sees that
\begin{equation}
\label{equss3}
\sss^*=\{0,\alpha,\alpha+2,\alpha+4,\ldots,\gamma-c-2,\gamma-b+2,\gamma-b+4,\gamma-b+6,\ldots,\beta-2,\beta\}
\end{equation}

Now set
\begin{equation}
\label{equaa2}
\aaa:=\oo_C\langle 1,t^2\rangle
\end{equation}
We have that $\aaa$ has degree $2$ at $\infty$ and 0 elsewhere but $P$. Moreover,
$$
\aaa_P=kt^2\oplus kt^{\gamma-c}\oplus \op
$$
hence $\dim(\aaa_P/\oo_P)=2$ and $\deg(\aaa)=4$. Therefore $\gon(C)\leq 4$ which, as seen above, is enough for our purposes.

If $a\neq\gamma-2$, all odd numbers are in $\kk$, $b$ is even and $d$ is odd, and, also, all positive odd numbers smaller than $\gamma$ are not in $\sss$, so $\alpha$ is even. We claim that $\alpha>\gamma/2$. In fact, otherwise, since $\gamma-\alpha$, which is not in $\kk$, is even, implies $\alpha\in A_1$; besisdes, $\gamma-\alpha<n\alpha<\gamma-1$ for a certain integer $n$, so $\gamma-n\alpha<\alpha$ is not in $\kk$, which contradicts the fact that $\alpha\in A_1$ and this proves the claim.  Hence
\begin{equation}
\label{equs13}
\sss^*=\{0,\alpha,\alpha+2,\alpha+4,\ldots,\gamma-a-2,\beta\}
\end{equation}
Taking $\aaa$ as in (\ref{equaa2}), we have that $\aaa$ has degree $2$ at $\infty$ and 0 elsewhere but $P$. Moreover,
$$
\aaa_P=kt^2\oplus kt^{\gamma-a}\oplus \op
$$
hence $\dim(\aaa_P/\oo_P)=2$ and $\deg(\aaa)=4$.

\noindent{\bf Case 2.} $\gamma$ is odd. \\
If $a=\gamma-1$, then all even numbers are in $\kk$, $b$ and $d$ are odd, and, also, all positive odd numbers smaller than $\gamma$ are not in $\sss$; so $\alpha$ is even. In this case there is no restriction to $\alpha$ compared to $\gamma/2$ and one can check that $\sss^*$ can be written exactly as in (\ref{equss3}).

If $a\neq\gamma-1$, all odd numbers are in $\kk$, $b$ is even and $d$ is odd, and, also, all positive even numbers smaller than $\gamma$ are not in $\sss$; so $\alpha$ is odd. Besides $\alpha>\gamma/2$ since $2\alpha>\gamma$, and $\sss^*$ can be written exactly as in (\ref{equs13}). Therefore one may use (\ref{equaa2}) to compute gonality $4$.

To prove sufficiency for $r= 3$, we may assume $\alpha\geq 4$. In fact, if $\alpha=1$ the $C=\pum$ which has gonality $1$, if $\alpha=2$, $C$ is hyperelliptic, and if $\alpha=3$, then $\oo\langle 1,t^3\rangle$ computes gonality $3$ for $C$. In all cases, $C'$ lie on a surface scroll owing to Theorem \ref{scroll2}. Therefore, $\gamma-1,\gamma-2,\gamma-3$ are in $\kk$ and belong to different subsets of the partition of $A$ and, in particular, all multiples of $3$ smaller than $\gamma$ are in $\kk$. We have three options: $\gamma=3n$, $\gamma=3n+1$ or $\gamma=3n+2$ for $n\in\mathbb{N}$. If $\gamma=3k$, the symmetric of a multiple of $3$ is a multiple of $3$ as well, thus no multiples of $3$ are in $S^*$. If $b$ (resp. $d$) is the first positive integer in $\sss$ congruent to $1$ (resp. $2$) mod $3$, it is easily seen that
$$
S^*=\{0,b,b+3,b+6,\ldots,\gamma-2\}\cup\{d,d+3,d+6,\ldots,\gamma-1\}\cup\{\beta\}
$$
Set $\aaa:=\oo_{C}\langle 1,t^3\rangle$, then $\aaa$ has degree $0$ elsewhere but at $P$ and $\infty$, where it has degree $3$. Besides $\aaa_P=kt^3\oplus \op$, so $\aaa$ has degree $1$ at $P$ and $\deg(\aaa)=\gon(C)=4$.

If $\gamma=3n+1$ or $\gamma=3n+2$ for $n\in\mathbb{N}$, the proof is analogous.

Now, if $r>3$, we may assume $\alpha\geq5$, because if $\alpha=4$, then $\aaa:=\oo_{C}\langle 1,t^4\rangle$ has degree $4$ and that is what we need. So if $\alpha\geq 5$, then $\gamma-1,\gamma-2,\gamma-3,\gamma-4\in\kk$ and they are in different subsets of the partition of $A$, but this cannot happen.

\medskip

Conversely, assume $\gon(C)=4$ and let us prove that $C'$ lies on a threefold scroll but not on a surface scroll. By \cite[p. 10]{FM}, the gonality of $C$ is computed by a sheaf 
$$
\aaa:=\mathcal{O}\langle 1,t^n\rangle
$$ 
for some nonzero $n\in\zz$. Note that
\begin{equation}
\label{equdin}
\deg_{\infty}\aaa=
\begin{cases}
n  & \text{if}\ n>0 \\ 
0  & \text{if}\ n\leq 0 
\end{cases}
\end{equation}
besides,
\begin{equation}
\label{equdep}
\deg_P(\aaa)=\#(v(\aaa_P)\setminus\sss)=\#\{s\in\sss\ |\ s+n\not\in\sss\}
\end{equation}
and the degree of $\aaa$ vanishes elsewhere. In particular, if $n$ is positive, then $1\leq n\leq 4$ since $C$ has gonality $4$.

For the remainder we will analyze the possibilities for $\sss$ according to each $n$. For this, we will split the semigroup into blocks of consecutive integers, that is, write 
$$
{\rm S}=\{0\}\cup B_1\cup\ldots\cup
B_b\cup\{n\in\mathbb{N}\,|\,n\geq\beta\}
$$  
where 
$$
B_i=\{s_i,s_i+1,\ldots, s_i+l_i\}
$$ 
with $s_i+l_i+1\not\in\sss$ and also $s_i<s_j$ if $i<j$. That is, $B_i$ is the $i$-th
block of consecutive integers in ${\rm S}$, and $b$ is the number of blocks of positive integers in $\sss$ smaller than the conductor $\beta$.

\noindent{\bf Case 1.} $n=1$. \\
If so, $\aaa=\mathcal{O}\langle 1,t\rangle$, $\deg_{\infty}(\aaa)=1$ and thus $\deg_P(\aaa)=3$. From (\ref{equdep}) it is easily seen that $\deg_P(\aaa)=b+1$, so $b=2$, that is, $\sss$ has two blocks.

To compute $\kk^*$ by means of $\sss^*$ we draw a picture in the following way: 
\begin{itemize}
\item[(a)] each number from $0$ to $\gamma$ corresponds to a circle;
\item[(b)] in the first row, one goes from $0$ to $\lfloor\gamma/2\rfloor$ ordered from left to right; 
\item[(c)] in the second row, one goes from $\lceil\gamma/2\rceil$ to $\gamma$ ordered from right to left;
\item[(d)] the elements of $\sss$ are black, the ones of $\kk\setminus\sss$ are double-circled, and the remaining numbers are white.
\end{itemize}
 In other words, the picture is sketched in a way that $a$ and $\gamma-a$ are in the same column, for any $a$, so that $\kk^*$ can be easily computed. Just note also that $a$ and $\gamma-a$ cannot be both in $\sss$, otherwise $\gamma$ would be in $\sss$, which cannot happen. 

There are two possibilities for $\sss^*$ (and $\kk^*$). The first one is

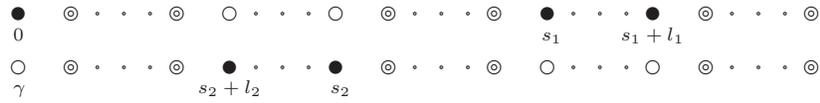
\begin{picture}(40,10)
\put(10,0){\circle*{5}}
\put(8,-10){\scriptsize$0$}
\put(30,0){\circle{5}}
\put(30,0){\circle{2}}
\put(40,0){\circle{0.7}}
\put(50,0){\circle{0.7}}
\put(60,0){\circle{0.7}}
\put(70,0){\circle{5}}
\put(70,0){\circle{2}}
\put(90,0){\circle{5}}
\put(100,0){\circle{0.7}}
\put(110,0){\circle{0.7}}
\put(120,0){\circle{0.7}}
\put(130,0){\circle{5}}
\put(150,0){\circle{5}}
\put(150,0){\circle{2}}
\put(160,0){\circle{0.7}}
\put(170,0){\circle{0.7}}
\put(180,0){\circle{0.7}}
\put(190,0){\circle{5}}
\put(190,0){\circle{2}}
\put(210,0){\circle*{5}}
\put(208,-10){\scriptsize$s_1$}
\put(220,0){\circle{0.7}}
\put(230,0){\circle{0.7}}
\put(240,0){\circle{0.7}}
\put(250,0){\circle*{5}}
\put(238,-10){\scriptsize$s_1+l_1$}
\put(270,0){\circle{5}}
\put(270,0){\circle{2}}
\put(280,0){\circle{0.7}}
\put(290,0){\circle{0.7}}
\put(300,0){\circle{0.7}}
\put(310,0){\circle{5}}
\put(310,0){\circle{2}}

\put(10,-20){\circle{5}}
\put(8,-30){\scriptsize$\gamma$}
\put(30,-20){\circle{5}}
\put(30,-20){\circle{2}}
\put(40,-20){\circle{0.7}}
\put(50,-20){\circle{0.7}}
\put(60,-20){\circle{0.7}}
\put(70,-20){\circle{5}}
\put(70,-20){\circle{2}}
\put(90,-20){\circle*{5}}
\put(78,-30){\scriptsize$s_2+l_2$}
\put(100,-20){\circle{0.7}}
\put(110,-20){\circle{0.7}}
\put(120,-20){\circle{0.7}}
\put(130,-20){\circle*{5}}
\put(128,-30){\scriptsize$s_2$}
\put(150,-20){\circle{5}}
\put(150,-20){\circle{2}}
\put(160,-20){\circle{0.7}}
\put(170,-20){\circle{0.7}}
\put(180,-20){\circle{0.7}}
\put(190,-20){\circle{5}}
\put(190,-20){\circle{2}}
\put(210,-20){\circle{5}}
\put(220,-20){\circle{0.7}}
\put(230,-20){\circle{0.7}}
\put(240,-20){\circle{0.7}}
\put(250,-20){\circle{5}}
\put(270,-20){\circle{5}}
\put(270,-20){\circle{2}}
\put(280,-20){\circle{0.7}}
\put(290,-20){\circle{0.7}}
\put(300,-20){\circle{0.7}}
\put(310,-20){\circle{5}}
\put(310,-20){\circle{2}}
\end{picture}

\vspace{4em}

Therefore we can split $\kk^*$ into three subsets with common difference $1$, namely:
\begin{align*}
A_1&=\{0,1,\ldots,\gamma-(s_2+l_2+1)\} \\
A_2&=\{\gamma-s_2+1,\gamma-s_2+2,\ldots,\gamma-(s_1+l_1+1)\}\\
A_3&=\{\gamma-s_1+1,\gamma-s_1+2,\ldots,\gamma-1\}
\end{align*}
so $C'$ lies on a threefold scroll.

The second possibility is

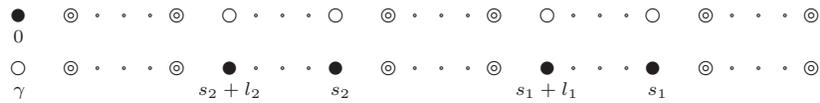
\begin{picture}(40,10)
\put(10,0){\circle*{5}}
\put(8,-10){\scriptsize$0$}
\put(30,0){\circle{5}}
\put(30,0){\circle{2}}
\put(40,0){\circle{0.7}}
\put(50,0){\circle{0.7}}
\put(60,0){\circle{0.7}}
\put(70,0){\circle{5}}
\put(70,0){\circle{2}}
\put(90,0){\circle{5}}
\put(100,0){\circle{0.7}}
\put(110,0){\circle{0.7}}
\put(120,0){\circle{0.7}}
\put(130,0){\circle{5}}
\put(150,0){\circle{5}}
\put(150,0){\circle{2}}
\put(160,0){\circle{0.7}}
\put(170,0){\circle{0.7}}
\put(180,0){\circle{0.7}}
\put(190,0){\circle{5}}
\put(190,0){\circle{2}}
\put(210,0){\circle{5}}
\put(220,0){\circle{0.7}}
\put(230,0){\circle{0.7}}
\put(240,0){\circle{0.7}}
\put(250,0){\circle{5}}
\put(270,0){\circle{5}}
\put(270,0){\circle{2}}
\put(280,0){\circle{0.7}}
\put(290,0){\circle{0.7}}
\put(300,0){\circle{0.7}}
\put(310,0){\circle{5}}
\put(310,0){\circle{2}}

\put(10,-20){\circle{5}}
\put(8,-30){\scriptsize$\gamma$}
\put(30,-20){\circle{5}}
\put(30,-20){\circle{2}}
\put(40,-20){\circle{0.7}}
\put(50,-20){\circle{0.7}}
\put(60,-20){\circle{0.7}}
\put(70,-20){\circle{5}}
\put(70,-20){\circle{2}}
\put(90,-20){\circle*{5}}
\put(78,-30){\scriptsize$s_2+l_2$}
\put(100,-20){\circle{0.7}}
\put(110,-20){\circle{0.7}}
\put(120,-20){\circle{0.7}}
\put(130,-20){\circle*{5}}
\put(128,-30){\scriptsize$s_2$}
\put(150,-20){\circle{5}}
\put(150,-20){\circle{2}}
\put(160,-20){\circle{0.7}}
\put(170,-20){\circle{0.7}}
\put(180,-20){\circle{0.7}}
\put(190,-20){\circle{5}}
\put(190,-20){\circle{2}}
\put(210,-20){\circle*{5}}
\put(198,-30){\scriptsize$s_1+l_1$}
\put(220,-20){\circle{0.7}}
\put(230,-20){\circle{0.7}}
\put(240,-20){\circle{0.7}}
\put(250,-20){\circle*{5}}
\put(248,-30){\scriptsize$s_1$}
\put(270,-20){\circle{5}}
\put(270,-20){\circle{2}}
\put(280,-20){\circle{0.7}}
\put(290,-20){\circle{0.7}}
\put(300,-20){\circle{0.7}}
\put(310,-20){\circle{5}}
\put(310,-20){\circle{2}}
\end{picture}

\vspace{3.5em}

\noindent and again we split $\kk^*$ into three subsets with common difference $1$, which coincide with $A_1,A_2,A_3$ above, 
so $C'$ lies on a threefold scroll once more.

\noindent{\bf Case 2.} $n=2$. \\
If so, $\aaa=\mathcal{O}\langle 1,t^2\rangle$, $\deg_{\infty}(\aaa)=2$ and thus $\deg_P(\aaa)=2$. For the remainder set
$$
E:=\{s+2\not\in\sss\ |\  s\in\sss\}
$$
and recall, from (\ref{equdep}), that
$$
2=\deg_P(\aaa)=\#(E)
$$
Note that $2\in E$. Besides, if there exists a block $B_j=\{s_j,s_j+1,\ldots, s_j+l_j\}\subset\sss^*$ with $l_j\geq 1$, then  $s_j+l_j+1\in E$. Therefore, there is at most one such a block, and we must have $s_j\geq\lceil\gamma/2\rceil$ otherwise, there would exist another block $B_m$ with $l_m=2l_j$. Moreover, $s_i+l_i+2\in\sss$ for all $i$. Finally, $s_1>\lceil\gamma/2\rceil$, otherwise $\gamma\in\sss$, which cannot happen, and $s_1+2i\in\sss$ for every $i\in\nn$. So, if there is a block with more than one element, the generic picture is

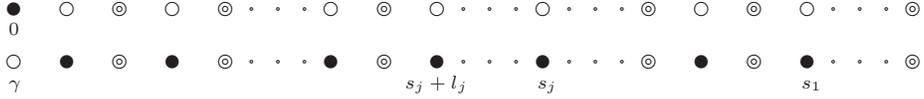
\begin{picture}(40,10)
\put(0,0){\circle*{5}}
\put(-2,-10){\scriptsize$0$}
\put(20,0){\circle{5}}
\put(40,0){\circle{5}}
\put(40,0){\circle{2}}
\put(60,0){\circle{5}}
\put(80,0){\circle{5}}
\put(80,0){\circle{2}}
\put(90,0){\circle{0.7}}
\put(100,0){\circle{0.7}}
\put(110,0){\circle{0.7}}
\put(120,0){\circle{5}}
\put(140,0){\circle{5}}
\put(140,0){\circle{2}}
\put(160,0){\circle{5}}
\put(170,0){\circle{0.7}}
\put(180,0){\circle{0.7}}
\put(190,0){\circle{0.7}}
\put(200,0){\circle{5}}
\put(210,0){\circle{0.7}}
\put(220,0){\circle{0.7}}
\put(230,0){\circle{0.7}}
\put(240,0){\circle{5}}
\put(240,0){\circle{2}}
\put(260,0){\circle{5}}
\put(280,0){\circle{5}}
\put(280,0){\circle{2}}
\put(300,0){\circle{5}}
\put(310,0){\circle{0.7}}
\put(320,0){\circle{0.7}}
\put(330,0){\circle{0.7}}
\put(340,0){\circle{5}}
\put(340,0){\circle{2}}

\put(0,-20){\circle{5}}
\put(-2,-30){\scriptsize$\gamma$}
\put(20,-20){\circle*{5}}
\put(40,-20){\circle{5}}
\put(40,-20){\circle{2}}
\put(60,-20){\circle*{5}}
\put(80,-20){\circle{5}}
\put(80,-20){\circle{2}}
\put(90,-20){\circle{0.7}}
\put(100,-20){\circle{0.7}}
\put(110,-20){\circle{0.7}}
\put(120,-20){\circle*{5}}
\put(140,-20){\circle{5}}
\put(140,-20){\circle{2}}
\put(160,-20){\circle*{5}}
\put(148,-30){\scriptsize$s_j+l_j$}
\put(170,-20){\circle{0.7}}
\put(180,-20){\circle{0.7}}
\put(190,-20){\circle{0.7}}
\put(200,-20){\circle*{5}}
\put(198,-30){\scriptsize$s_j$}
\put(210,-20){\circle{0.7}}
\put(220,-20){\circle{0.7}}
\put(230,-20){\circle{0.7}}
\put(240,-20){\circle{5}}
\put(240,-20){\circle{2}}
\put(260,-20){\circle*{5}}
\put(280,-20){\circle{5}}
\put(280,-20){\circle{2}}
\put(300,-20){\circle*{5}}
\put(298,-30){\scriptsize$s_1$}
\put(310,-20){\circle{0.7}}
\put(320,-20){\circle{0.7}}
\put(330,-20){\circle{0.7}}
\put(340,-20){\circle{5}}
\put(340,-20){\circle{2}}

\end{picture}

\vspace{3.5em}

Therefore we can split $\kk^*$ into three subsets with common difference $2$, namely:
\begin{align*}
A_1&=\{0,2,\ldots,\gamma-(s_j+l_j+1)\} \\
A_2&=\{\gamma-s_j+1,\gamma-s_j+3,\ldots,\gamma-1\ \text{or}\ \gamma-2\}\\
A_3&=\{\gamma-s_1+2,\gamma-s_1+4,\ldots,\gamma-1\ \text{or}\ \gamma-2\}
\end{align*}
so $C'$ lies on a threefold scroll.

If all the blocks have just one element, then there exists a unique $s_j$ such that $s_j+2\not\in\sss$, and the generic picture is

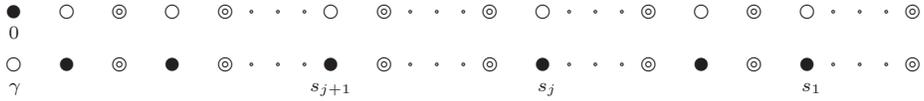
\begin{picture}(40,10)
\put(0,0){\circle*{5}}
\put(-2,-10){\scriptsize$0$}
\put(20,0){\circle{5}}
\put(40,0){\circle{5}}
\put(40,0){\circle{2}}
\put(60,0){\circle{5}}
\put(80,0){\circle{5}}
\put(80,0){\circle{2}}
\put(90,0){\circle{0.7}}
\put(100,0){\circle{0.7}}
\put(110,0){\circle{0.7}}
\put(120,0){\circle{5}}
\put(140,0){\circle{5}}
\put(140,0){\circle{2}}
\put(150,0){\circle{0.7}}
\put(160,0){\circle{0.7}}
\put(170,0){\circle{0.7}}
\put(180,0){\circle{5}}
\put(180,0){\circle{2}}
\put(200,0){\circle{5}}
\put(210,0){\circle{0.7}}
\put(220,0){\circle{0.7}}
\put(230,0){\circle{0.7}}
\put(240,0){\circle{5}}
\put(240,0){\circle{2}}
\put(260,0){\circle{5}}
\put(280,0){\circle{5}}
\put(280,0){\circle{2}}
\put(300,0){\circle{5}}
\put(310,0){\circle{0.7}}
\put(320,0){\circle{0.7}}
\put(330,0){\circle{0.7}}
\put(340,0){\circle{5}}
\put(340,0){\circle{2}}

\put(0,-20){\circle{5}}
\put(-2,-30){\scriptsize$\gamma$}
\put(20,-20){\circle*{5}}
\put(40,-20){\circle{5}}
\put(40,-20){\circle{2}}
\put(60,-20){\circle*{5}}
\put(80,-20){\circle{5}}
\put(80,-20){\circle{2}}
\put(90,-20){\circle{0.7}}
\put(100,-20){\circle{0.7}}
\put(110,-20){\circle{0.7}}
\put(120,-20){\circle*{5}}
\put(112,-30){\scriptsize$s_{j+1}$}
\put(140,-20){\circle{5}}
\put(140,-20){\circle{2}}
\put(150,-20){\circle{0.7}}
\put(160,-20){\circle{0.7}}
\put(170,-20){\circle{0.7}}
\put(180,-20){\circle{5}}
\put(180,-20){\circle{2}}
\put(200,-20){\circle*{5}}
\put(198,-30){\scriptsize$s_j$}
\put(210,-20){\circle{0.7}}
\put(220,-20){\circle{0.7}}
\put(230,-20){\circle{0.7}}
\put(240,-20){\circle{5}}
\put(240,-20){\circle{2}}
\put(260,-20){\circle*{5}}
\put(280,-20){\circle{5}}
\put(280,-20){\circle{2}}
\put(300,-20){\circle*{5}}
\put(298,-30){\scriptsize$s_1$}
\put(310,-20){\circle{0.7}}
\put(320,-20){\circle{0.7}}
\put(330,-20){\circle{0.7}}
\put(340,-20){\circle{5}}
\put(340,-20){\circle{2}}

\end{picture}

\vspace{3.5em}

Therefore we can split $\kk^*$ into three subsets with common difference $2$, namely:
\begin{align*}
A_1&=\{0,2,\ldots\ldots,\gamma-s_j-2\ \text{or}\ \gamma-1\ \text{or}\ \gamma-2\} \\
A_2&=\{\gamma-s_{j+1}+2,\gamma-s_{j+1}+4,\ldots\dots,\gamma-s_j-2\ \text{or}\ \gamma-1\ \text{or}\ \gamma-2\}\\
A_3&=\{\gamma-s_1+2,\gamma-s_1+4,\ldots\ldots, \gamma-1\ \text{or}\ \gamma-2\}
\end{align*}
so $C'$ lies on a threefold scroll.

\noindent{\bf Case 3.} $n=3$. \\
If so, $\aaa=\mathcal{O}\langle 1,t^3\rangle$, $\deg_{\infty}(\aaa)=3$ and thus 
$$
1=\deg_P(\aaa)=\#(E:=\{s+3\not\in\sss\ |\  s\in\sss\})
$$
We may assume $\alpha\geq 4$, otherwise $\deg(\aaa)=3$ and $C$ is trigonal. Therefore $3\in E$. This forces that $s+3\in\sss$ for every $s\in\sss\setminus\{0\}$, so the blocks $B_i$ have at most $2$ elements. Let $B_j$ be the first block with $2$ elements. Clearly, $s_j\geq\lceil\gamma/2\rceil$ otherwise, there would exist another block $3$ elements. The generic picture for this case is 

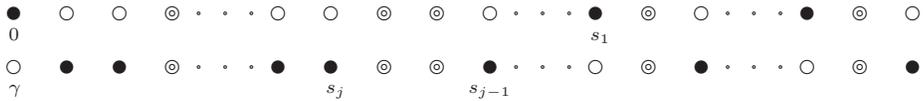
\begin{picture}(40,10)
\put(0,0){\circle*{5}}
\put(-2,-10){\scriptsize$0$}
\put(20,0){\circle{5}}
\put(40,0){\circle{5}}
\put(60,0){\circle{5}}
\put(60,0){\circle{2}}
\put(70,0){\circle{0.7}}
\put(80,0){\circle{0.7}}
\put(90,0){\circle{0.7}}
\put(100,0){\circle{5}}
\put(120,0){\circle{5}}
\put(140,0){\circle{5}}
\put(140,0){\circle{2}}
\put(160,0){\circle{5}}
\put(160,0){\circle{2}}
\put(180,0){\circle{5}}
\put(190,0){\circle{0.7}}
\put(200,0){\circle{0.7}}
\put(210,0){\circle{0.7}}
\put(220,0){\circle*{5}}
\put(218,-10){\scriptsize$s_1$}
\put(240,0){\circle{5}}
\put(240,0){\circle{2}}
\put(260,0){\circle{5}}
\put(270,0){\circle{0.7}}
\put(280,0){\circle{0.7}}
\put(290,0){\circle{0.7}}
\put(300,0){\circle*{5}}
\put(320,0){\circle{5}}
\put(320,0){\circle{2}}
\put(340,0){\circle{5}}

\put(0,-20){\circle{5}}
\put(-2,-30){\scriptsize$\gamma$}
\put(20,-20){\circle*{5}}
\put(40,-20){\circle*{5}}
\put(60,-20){\circle{5}}
\put(60,-20){\circle{2}}
\put(70,-20){\circle{0.7}}
\put(80,-20){\circle{0.7}}
\put(90,-20){\circle{0.7}}
\put(100,-20){\circle*{5}}
\put(120,-20){\circle*{5}}
\put(118,-30){\scriptsize$s_j$}
\put(140,-20){\circle{5}}
\put(140,-20){\circle{2}}
\put(160,-20){\circle{5}}
\put(160,-20){\circle{2}}
\put(180,-20){\circle*{5}}
\put(172,-30){\scriptsize$s_{j-1}$}
\put(190,-20){\circle{0.7}}
\put(200,-20){\circle{0.7}}
\put(210,-20){\circle{0.7}}
\put(220,-20){\circle{5}}
\put(240,-20){\circle{5}}
\put(240,-20){\circle{2}}
\put(260,-20){\circle*{5}}
\put(270,-20){\circle{0.7}}
\put(280,-20){\circle{0.7}}
\put(290,-20){\circle{0.7}}
\put(300,-20){\circle{5}}
\put(320,-20){\circle{5}}
\put(320,-20){\circle{2}}
\put(340,-20){\circle*{5}}

\end{picture}

\vspace{3.5em}

\noindent and we can split $\kk^*$ into three subsets with common difference $3$, namely:
\begin{align*}
A_1&=\{0,3,\ldots,\gamma-1\ \text{or}\ \gamma-2\ \text{or}\ \gamma-3\} \\
A_2&=\{\gamma-s_j+2,\gamma-s_j+4,\ldots,\gamma-1\ \text{or}\ \gamma-2\ \text{or}\ \gamma-3\}\\
A_3&=\{s_1\ \text{or}\ s_{j-1}\pm 1,\ldots,\gamma-1\ \text{or}\ \gamma-2\ \text{or}\ \gamma-3\}
\end{align*}
so $C'$ lies on a threefold scroll. The case without any block with $2$ elements is similar. 

\noindent{\bf Case 4.} $n=4$. \\
If so, $\aaa=\mathcal{O}\langle 1,t^4\rangle$, $\deg_{\infty}(\aaa)=4$ and thus 
$$
0=\deg_P(\aaa)=\#(E:=\{s+4\not\in\sss\ |\  s\in\sss\})
$$
Note again that $\alpha\geq 4$ because otherwise $C$ would be trigonal. So actually $\alpha=4$ since $E=\emptyset$, then every multiple of $4$ is in $\sss$. Let $a$ be such that $\gamma\equiv a\ ({\rm mod}\ 4)$.  Clearly, $\gamma$ is not a multiple of $4$, hence one may write $\{1,2,3\}=\{a,b,c\}$. We have that $\sss^*$ admits at most two elements, say $b'$ and $c'$, for which  $b'\equiv b\ ({\rm mod}\ 4)$ and  $c'\equiv c\ ({\rm mod}\ 4)$. We will consider three cases. The first one is where such $b'$ and $c'$ do not exist. If so, we can split $\kk^*$ into three subsets with common difference $4$, namely:
\begin{align*}
A_1&=\{0,4,8,\ldots,\gamma-a\} \\
A_2&=\{b,b+4,b+8,\ldots,\gamma-c\}\\
A_3&=\{c,c+4,c+8,\ldots,\gamma-b\}
\end{align*}
so $C'$ lies on a threefold scroll.

If there is such a $b'$, but no $c'$ in $\sss^*$, then we can split $\kk^*$ into three subsets with common difference $4$, namely:
\begin{align*}
A_1&=\{0,4,8,\ldots,\gamma-a\} \\
A_2&=\{c,c+4,c+8,\ldots,\gamma-b\}\\
A_3&=\{\gamma-b'+4,\gamma-b'+8,\ldots,\gamma-c\}
\end{align*}
so $C'$ lies on a threefold scroll.

Finally, if there are such $b'$ and $c'$  in $\sss^*$, then we can split $\kk^*$ into three subsets with common difference $4$, namely:
\begin{align*}
A_1&=\{0,4,8,\ldots,\gamma-a\} \\
A_2&=\{\gamma-b'+4,\gamma-b'+8,\ldots,\gamma-c\}\\
A_3&=\{\gamma-c'+4,\gamma-c'+8,\ldots,\gamma-b\}
\end{align*}
so $C'$ lies on a threefold scroll.

Now, if $n<0$, for the sake of convenience, write $n=-m$. Since $\deg_{\infty}\aaa=0$, we have that
\begin{equation*}
4=\deg_P(\aaa)=\#(E:=\{s\in\sss\ |\ s-m\not\in\sss\})
\end{equation*}
Besides, $-m$ and $\gamma$ are in $E$, so
\begin{equation}
\label{equemg}
2=\#(E\setminus\{-m,\gamma\})
\end{equation}
Now consider the set $B:=\{b\in\nn\,|\, \gamma-m+1\leq b\leq\gamma-1\}$. First we claim that
\begin{equation}
\label{equgn3}
\#(B\cap\sss)\geq m-3
\end{equation}
In fact, the integers from $\gamma+1$ to $\gamma+m-1$ are all in $\sss$ since $\gamma+1$ is the conductor of $\sss$. Hence any element of $B$ can be written as $s-m$ for a certain $s\in \sss$; since $\#(B)=m-1$, we may use (\ref{equemg}) to conclude that $\#(B\cap\sss)\geq (m-1)-2$ and the claim follows. On the other hand, if $b\in B\cap\sss$ is such that $m\nmid b$, then clearly $b-rm\in E\setminus\{-m,\gamma\}$ for a certain $r$. Hence we may refine (\ref{equgn3}) as
\begin{equation}
\label{equgnr}
\#(B\cap\sss)\geq m-3+\#(\{b\in B\cap\sss\,|\,m\nmid b\})
\end{equation}
but this clearly implies that $m\leq 4$. Indeed, if $m\geq 5$ then, by (\ref{equgn3}), we have that $\#(B\cap\sss)\geq 2$, so there is a nonmultiple of $m$ in $B\cap\sss$. But by (\ref{equgnr}), actually, $\#(B\cap\sss)\geq 3$, and one finds at least two nonmultiples of $m$ in $B\cap\sss$ so $\#(B\cap\sss)\geq 4$. This yields at least three nonmultiples of $m$ in $B\cap\sss$ and using (\ref{equgnr}) once again we have that $\#(B\cap\sss)\geq m>\#(B)$, which is a contradiction.

Thus $n\geq -4$, and we leave to the reader verifying that the possible semigroups for the cases where $n$ is $-1$, $-2$, $-3$ and $-4$ are, respectively, the same as those corresponding to the case $n$ is $1$, $2$, $3$ and $4$, which were already analyzed.
\end{proof}

The results on rational monomial curves suggest that maybe is worth pursueing the question whether the geometric characterization of gonality by means of scrolls can pass from canonical curves to non-Gorenstein ones. This would certainly involve a careful study of syzygies, though the reader should note the difficulty of adapting the arguments of, for instance,  \cite{Sc,CE,BS} when the dualizing sheaf of the curve fails to be a bundle. Rather, here we've combined two completely different techniques, namely, intersection theory along with semigroup of values, as a somewhat initial approach to this problem.


\begin{thebibliography}{BDF2}
\bibitem[A]{A} E. Arrondo, {\it A home-made Hartshorne-Serre correspondence}, arXiv: math/0610015v1 [math.AG] 30 Sep 2006.
\bibitem[ACGH]{ACGH} E. Arbarello, M. Cornalba,  P. A. Griffiths, and J. Harris, \emph{Geometry of algebraic curves}, Springer-Verlag (1985).
\bibitem[B]{B} D. W. Babbage, 
{\it A note on the quadrics through a canonical curve}, J. Lodon Math. Soc. 14 (1939), 310--315.
\bibitem[BF]{BF}     V. Barucci, R. Fr\"oberg, {\it One-Dimensional Almost Gorenstein Rings}, Journal of Algebra 188 (1997), 418--442.
\bibitem[BS]{BS} M. Brundu, G. Sacchiero, {\it Stratification of the moduli space of fourgonal curves}, Proc. Edinb. Math. Soc., Vol. 57, Issue 03, 631 - 686 (2014)
\bibitem[CE]{CE} G. Casnati, T. Ekedahl, {\it Covers of algebraic varieties. I. A general structure theorem, covers of degree 3,4 and Enriques
surfaces}, J. Algebraic Geom. 5 (1996), no. 3, 439 - 460.
\bibitem[Cp]{Cp}  M. Coppens, \emph{Free Linear Systems on Integral Gorenstein Curves}, Journal of Algebra 145 (1992), 209--218.
\bibitem[EH]{EH} D. Eisenbud, J. Harris, \emph{On Varieties of Minimal Degree}, Proceedings of Symposia in Pure Mathematics 46 (1887), 3--13.
\bibitem[En]{En} F. Enriques, {\it Sulle curve canoniche di genera $p$ cello spazio a $p-1$ dimensioni},
Rend. Accad. Sci. Ist. Bologna, 23 (1919), 80--82.
\bibitem[FM]{FM} L. Feital. R. V. Martins, {\it Gonality of non-Gorenstein curves of genus five},  Bull. Braz. Math. Soc., 45(4) (2014), 1--22.
\bibitem[KM]{KM} S. L. Kleiman, R. V. Martins, {\it The Canonical Model of a Singular Curvee}, Geometria Dedicata. 139 (2009), 139-166.
\bibitem[L]{L} D. Lara, \emph{Curvas com Modelos Can\^onicos em Scrolls}, Ph. D. Thesis UFMG (2014).
\bibitem[M1]{M1}   R. V. Martins, \emph{On Trigonal Non-Gorenstein Curves with Zero Maroni Invariant}, Journal of Algebra 275 (2004), 453--470.
\bibitem[M2]{M2}   R. V. Martins, {\it Trigonal non-Gorenstein curves}, J. Pure Appl. Algebra, 209 (2007), 873--882.
\bibitem[Mr]{Mr} R. M. Mir\'o-Roig, \emph{The Representation Type of Rational Normal Scrolls}, Rend. Circ. Mat. Palermo 62 (2012), 153--164.
\bibitem[N]{N} M. Noether,
{\it \"Uber die invariante Darstellung algebraicher Funktionen},
Math. Ann. 17 (1880), 263--284.
\bibitem[P]{P} K. Petri,
{\it \"Uber die invariante {D}arstellung algebraischer Funktionen eiener  Ver\"anderlichen},
Math. Ann. 88 (1922), 242--289.
\bibitem[R]{R}     M. Rosenlicht, {\it Equivalence Relations on Algebraic Curves}, Annals of Mathematics 56 (1952), 169--191
\bibitem[Rd]{Rd}  M. Reid, \emph{Chapters on Algebraic Surfaces}.
\bibitem[RS]{RS}      R. Rosa, K.-O. St\"ohr, {\it Trigonal Gorenstein Curves}, Journal of Pure and Applied Algebra 174 (2002), 187-205.
\bibitem[S1]{S1}        K.-O. St\"ohr, {\it On the Poles of Regular Differentials of Singular Curves}, Boletim da Sociedade Brasileira de Matem\'atica 24 (1993), 105--135.
\bibitem[S2]{S2}        K.-O. St\"ohr, {\it Hyperelliptic Gorenstein Curves}, Journal of Pure and Applied Algebra 135 (1999), 93--105.
\bibitem[Sc]{Sc} F.-O. Schreyer, {\it Syzygies of Canonical Curves and Special Linear Series}, Mathematische Annalen 275 (1986), 105--137.
\bibitem[SV]{SV}        K.-O. St\"ohr, P. Viana, {\it Weierstrass gap sequences and moduli varieties of trigonal curves},  Journal of Pure and Applied Algebra 81 (1992), 63--82.
\end{thebibliography}
\end{document}